\documentclass{article}
\usepackage[margin=1in]{geometry}
\usepackage[utf8]{inputenc}
\usepackage{mathrsfs}
\usepackage{amsmath, amsfonts, amsthm, bm, varwidth, hyperref, xfrac, textgreek, mathtools, booktabs, tabularx, array, float, xcolor, mdframed, amssymb}
\usepackage{caption}
\usepackage{algorithm}
\usepackage[noend]{algpseudocode}
\usepackage{hyperref}
\usepackage{enumitem}
\usepackage{multirow}
\usepackage{siunitx}
\usepackage{scalerel}
\usepackage{graphicx}
\usepackage{pdflscape}
\usepackage{tikz}
\usetikzlibrary{backgrounds,matrix,fit,decorations.pathreplacing,calc,positioning,calligraphy}
\pgfsetlayers{main,background}
\usepackage{arydshln}
\usepackage{extarrows}
\usepackage[backend=biber,style=alphabetic,maxalphanames=5,maxnames=8]{biblatex}
\providecommand{\nopp}{}
\newcolumntype{C}{>{\raggedleft\arraybackslash}X} 
\usepackage{mleftright}
\mleftright
 
\usepackage{latexml}
 
\newtheorem{theorem}{Theorem}[section]
\newtheorem{lemma}[theorem]{Lemma}

\theoremstyle{definition}

\newtheorem{remark}[theorem]{Remark}
\newtheorem{example}[theorem]{Example}
\newtheorem{question}[theorem]{Question}
\newtheorem*{theorem*}{Theorem}
\newtheorem*{conjecture*}{Conjecture}
 
\newtheoremstyle{break}
  {}
  {}%
  {}
  {}%
  {\bfseries}
  {.}%
  {\newline}
  {}%
\theoremstyle{break}
\newtheorem{construction}[theorem]{CONSTRUCTION}
\BeforeBeginEnvironment{construction}{\begin{mdframed}}
\AfterEndEnvironment{construction}{\end{mdframed}}
 
\newcommand*{\horzbar}{\rule[.5ex]{2.9ex}{0.5pt}}
 
\DeclarePairedDelimiterX{\set}[1]{\{}{\}}{\setargs{#1}}
\NewDocumentCommand{\setargs}{>{\SplitArgument{1}{;}}m}
{\setargsaux#1}
\NewDocumentCommand{\setargsaux}{mm}
{\IfNoValueTF{#2}{#1} {#1\nonscript\:\delimsize\vert\allowbreak\nonscript\:\mathopen{}#2}}%
\allowdisplaybreaks
 
\algrenewcommand\textproc{}
\title{Hodge Structures in Sextic Fourfolds Equipped with an Involution}
\author{Benjamin E. \textsc{Diamond} \\
	\scriptsize \href{mailto:benediamond@gmail.com}{\texttt{benediamond@gmail.com}}
}
\date{}
\newif\ifdraftfullpagefigure
\draftfullpagefigurefalse

\newcommand{\fullpagefigureforpdf}{%
  \ifdim\pagetotal>0pt
    \clearpage
  \fi
  \thispagestyle{empty}%
  \null
  \begin{tikzpicture}[remember picture, overlay]
    \ifdraftfullpagefigure
      \node[anchor=center] at (current page.center) {%
  \begin{tikzpicture}[x=1in,y=1in]
\path[use as bounding box] (0,0) rectangle (11,8.5);

\tikzset{
  complex cell/.style={
    line width=0.5pt,
    fill=white,
    align=center,
    inner sep=4pt,
    minimum height=0.35in,
    minimum width=1.4in,
  },
  complex arrow/.style={->, line width=0.6pt},
  complex dots/.style={
    minimum height=0.35in,
    minimum width=0.7in,
  }
}
\def\topend{8.2}
\def\leftend{0.85}
\def\horizontal{1.71}
\def\vertical{1.45}
\def\truncation{2.45}

\newcommand{\complexcell}[5]{%
  \node[complex cell] (#1) at (#2,#3)
    {$\check{C}^{#4}\!\left(\mathscr{U}, \Omega_{\mathbb{P}^n}^{#5}(\log X)\right)$};%
}

\node [minimum height=0.35in, minimum width=1.4in]  (c1nphantom) at (\leftend+\horizontal, \topend) {};
\node [below=8pt] at (c1nphantom.south) (c1nphantombox) {$\phantom{\left\{ -\delta d\vartheta^{(0)} = (-1) \cdot d\omega^{(1)} \right\}}$};
\node [minimum height=0.35in, minimum width=1.4in]  (c1nmonephantom) at (\leftend+\horizontal, \topend-\vertical) {};
\node [below=8pt] at (c1nmonephantom.south) (c1nmonephantombox) {};
\draw[complex arrow,|->,shorten <=3pt,shorten >=3pt] ([xshift=15pt]c1nmonephantombox.north) -- ([xshift=15pt]c1nphantombox.south); 

\node [minimum height=0.35in, minimum width=1.4in] (c2nmonephantom) at (\leftend+2*\horizontal, \topend-\vertical) {};
\node [below=8pt] at (c2nmonephantom.south) (c2nmonephantombox) {$\phantom{\left\{ -\delta d\vartheta^{(1)} = (-1)^2 \cdot d\omega^{(2)} \right\}}$};
\node [minimum height=0.35in, minimum width=1.4in] (c2nmtwophantom) at (\leftend+2*\horizontal, \topend-2*\vertical) {};
\node [below=8pt] at (c2nmtwophantom.south) (c2nmtwophantombox) {};
\draw[complex arrow,|->,shorten <=3pt,shorten >=3pt] ([xshift=15pt]c2nmtwophantombox.north) -- ([xshift=15pt]c2nmonephantombox.south); 

\node [minimum height=0.35in, minimum width=1.4in] (cqmponeppthreephantom) at (\leftend+\truncation*\horizontal+\horizontal, \topend-\truncation*\vertical) {};
\node [below=8pt] at (cqmponeppthreephantom.south) (cqmponeppthreephantombox) {$\phantom{\left\{-\delta d\vartheta^{(q - 2)} = (-1)^{q - 1} \cdot d\omega^{(q - 1)} \right\}}$};
\node [minimum height=0.35in, minimum width=1.4in] (cqmponepptwophantom) at (\leftend+\truncation*\horizontal+\horizontal, \topend-\truncation*\vertical-\vertical) {};
\node [below=8pt] at (cqmponepptwophantom.south) (cqmponepptwophantombox) {};
\draw[complex arrow,|->,shorten <=3pt,shorten >=3pt] ([xshift=15pt]cqmponepptwophantombox.north) -- ([xshift=15pt]cqmponeppthreephantombox.south); 

\node [minimum height=0.35in, minimum width=1.4in] (cqpptwophantom) at (\leftend+\truncation*\horizontal+2*\horizontal, \topend-\truncation*\vertical-\vertical) {};
\node [below=8pt] at (cqpptwophantom.south) (cqpptwophantombox) {$\phantom{\left\{ -\delta d \vartheta^{(q - 1)} = (-1)^q \cdot d\omega^{(q)} \right\}}$};
\node [minimum height=0.35in, minimum width=1.4in] (cqpponephantom) at (\leftend+\truncation*\horizontal+2*\horizontal, \topend-\truncation*\vertical-2*\vertical) {};
\node [below=8pt] at (cqpponephantom.south) (cqpponephantombox) {};
\draw[complex arrow,|->,shorten <=3pt,shorten >=3pt] ([xshift=15pt]cqpponephantombox.north) -- ([xshift=15pt]cqpptwophantombox.south); 

\complexcell{c0n}{\leftend}{\topend}{0}{n}
\complexcell{c1n}{\leftend+\horizontal}{\topend}{1}{n}
\complexcell{c0nmone}{\leftend}{\topend-\vertical}{0}{n - 1}
\complexcell{c1nmone}{\leftend+\horizontal}{\topend-\vertical}{1}{n - 1}
\complexcell{c0nmtwo}{\leftend}{\topend-2*\vertical}{0}{n - 2}

\complexcell{c2nmone}{\leftend+2*\horizontal}{\topend-\vertical}{2}{n - 1}
\complexcell{c1nmtwo}{\leftend+\horizontal}{\topend-2*\vertical}{1}{n - 2}
\complexcell{c2nmtwo}{\leftend+2*\horizontal}{\topend-2*\vertical}{2}{n - 2}

\node[complex dots] (toptruncate) at (\leftend+2*\horizontal, \topend) {$\cdots$};
\node[complex dots] (top2truncate) at (\leftend+3*\horizontal, \topend-\vertical) {$\cdots$};
\node[complex dots] (lefttruncate) at (\leftend, \topend-3*\vertical) {$\vdots$};
\node[complex dots] (left2truncate) at (\leftend + \horizontal, \topend-3*\vertical) {$\vdots$};
\node[complex dots] (left3truncate) at (\leftend + 2*\horizontal, \topend-3*\vertical) {$\vdots$};

\complexcell{cqmponeppthree}{\leftend+\truncation*\horizontal+\horizontal}{\topend-\truncation*\vertical}{q - 1}{p + 3}
\complexcell{cqmponepptwo}{\leftend+\truncation*\horizontal+\horizontal}{\topend-\truncation*\vertical-\vertical}{q - 1}{p + 2}
\complexcell{cqmponeppone}{\leftend+\truncation*\horizontal+\horizontal}{\topend-\truncation*\vertical-2*\vertical}{q - 1}{p + 1}
\complexcell{cqpptwo}{\leftend+\truncation*\horizontal+2*\horizontal}{\topend-\truncation*\vertical-\vertical}{q}{p + 2}
\complexcell{cqppone}{\leftend+\truncation*\horizontal+2*\horizontal}{\topend-\truncation*\vertical-2*\vertical}{q}{p + 1}

\node[complex dots] (lasttruncate) at (\leftend+\truncation*\horizontal+2*\horizontal, \topend-\truncation*\vertical) {$\cdots$};
\node[complex dots] (last2truncate) at (\leftend+\truncation*\horizontal+3*\horizontal, \topend-\truncation*\vertical-\vertical) {$\cdots$};
\node[complex dots] (last3truncate) at (\leftend+\truncation*\horizontal+3*\horizontal, \topend-\truncation*\vertical-2*\vertical) {$\cdots$};

\node[complex dots] (beforetruncate) at (\leftend+\truncation*\horizontal, \topend-\truncation*\vertical) {$\cdots$};
\node[complex dots] (before2truncate) at (\leftend+\truncation*\horizontal, \topend-\truncation*\vertical-\vertical) {$\cdots$};
\node[complex dots] (before3truncate) at (\leftend+\truncation*\horizontal, \topend-\truncation*\vertical-2*\vertical) {$\cdots$};

\node[complex dots] (undertruncate) at (\leftend+\truncation*\horizontal+\horizontal, \topend-\truncation*\vertical-3*\vertical) {$\vdots$};
\node[complex dots] (under2truncate) at (\leftend+\truncation*\horizontal+2*\horizontal, \topend-\truncation*\vertical-3*\vertical) {$\vdots$};

\draw[complex arrow] (cqmponeppthree.east) -- (lasttruncate.west) node[midway, above] {$\delta$};
\draw[complex arrow] (cqpptwo.east) -- (last2truncate.west) node[midway, above] {$\delta$};
\draw[complex arrow] (cqppone.east) -- (last3truncate.west) node[midway, above] {$\delta$};

\draw[complex arrow] (beforetruncate) -- (cqmponeppthree.west) node[midway, above] {$\delta$};
\draw[complex arrow] (before2truncate) -- (cqmponepptwo.west) node[midway, above] {$\delta$};
\draw[complex arrow] (before3truncate) -- (cqmponeppone.west) node[midway, above] {$\delta$};

\draw[complex arrow] (undertruncate) -- (cqmponeppone.south); 
\draw[complex arrow] (under2truncate) -- (cqppone.south); 

\draw[complex arrow] (c0n.east) -- (c1n.west) node[midway, above] {$\delta$};
\draw[complex arrow] (c0nmone.east) -- (c1nmone.west) node[midway, above] {$\delta$};
\draw[complex arrow] (c0nmtwo.east) -- (c1nmtwo.west) node[midway, above] {$\delta$};
\draw[complex arrow] (c1nmone.east) -- (c2nmone.west) node[midway, above] {$\delta$};
\draw[complex arrow] (c1nmtwo.east) -- (c2nmtwo.west) node[midway, above] {$\delta$};

\draw[complex arrow] (c0nmone.north) -- (c0n.south); 
\draw[complex arrow] (c0nmtwo.north) -- (c0nmone.south); 
\draw[complex arrow] (c1nmone.north) -- (c1n.south); 
\draw[complex arrow] (c1nmtwo.north) -- (c1nmone.south); 
\draw[complex arrow] (c2nmtwo.north) -- (c2nmone.south); 

\draw[complex arrow] (c1n.east) -- (toptruncate.west) node[midway, above] {$\delta$};
\draw[complex arrow] (c2nmone.east) -- (top2truncate.west) node[midway, above] {$\delta$};
\draw[complex arrow] (lefttruncate.north) -- (c0nmtwo.south); 
\draw[complex arrow] (left2truncate.north) -- (c1nmtwo.south); 
\draw[complex arrow] (left3truncate.north) -- (c2nmtwo.south); 

\draw[complex arrow] (cqmponepptwo.north) -- (cqmponeppthree.south); 
\draw[complex arrow] (cqmponeppone.north) -- (cqmponepptwo.south); 
\draw[complex arrow] (cqmponepptwo.east) -- (cqpptwo.west) node[midway, above] {$\delta$};
\draw[complex arrow] (cqmponeppone.east) -- (cqppone.west) node[midway, above] {$\delta$};
\draw[complex arrow] (cqppone.north) -- (cqpptwo.south); 

\node [below=8pt,fill=white,draw,dashed] at (c0n.south) (c0nbox) {$\left\{ \omega^{(0)} - \partial\vartheta^{(0)} \right\}$};
\node [below=8pt,fill=white,draw,dashed] at (c1n.south) (c1nbox) {$\left\{ -\delta \partial\vartheta^{(0)} = (-1)^1 \cdot \partial\omega^{(1)} \right\}$};
\node [below=8pt,fill=white,draw,dashed] at (c1nmone.south) (c1nmonebox) {$\left\{ \omega^{(1)} - \partial\vartheta^{(1)} \right\}$};
\node [below=8pt,fill=white,draw,dashed] at (c0nmone.south) (c0nmonebox) {$\left\{ \beta^{(0)} - \vartheta^{(0)} - \partial\zeta^{(0)} \right\}$};
\node [below=8pt,fill=white,draw,dashed] at (c0nmtwo.south) (c0nmtwobox) {$\left\{ \zeta^{(0)} \right\}$};
\node [below=8pt,fill=white,draw,dashed] at (c1nmtwo.south) (c1nmtwobox) {$\left\{ \beta^{(1)} - (-1)^1 \cdot \vartheta^{(1)} - \partial\zeta^{(1)} \right\}$};
\node [below=8pt,fill=white,draw,dashed] at (c2nmone.south) (c2nmonebox) {$\left\{ -\delta \partial\vartheta^{(1)} = (-1)^2 \cdot \partial\omega^{(2)} \right\}$};
\node [below=8pt,fill=white,draw,dashed] at (c2nmtwo.south) (c2nmtwobox) {$\left\{ \omega^{(2)} - \partial\vartheta^{(2)}\right\}$};

\node [below=8pt,fill=white,draw,dashed] at (cqmponepptwo.south) (cqmponepptwobox) {$\left\{ \omega^{(q - 1)} - \partial\vartheta^{(q - 1)} \right\}$};
\node [below=8pt,fill=white,draw,dashed] at (cqmponeppthree.south) (cqmponeppthreebox) {$\left\{-\delta \partial\vartheta^{(q - 2)} = (-1)^{q - 1} \cdot \partial\omega^{(q - 1)} \right\}$};
\node [below=8pt,fill=white,draw,dashed] at (cqppone.south) (cqppnebox) {$\left\{ \omega^{(q)} \right\}$};
\node [below=8pt,fill=white,draw,dashed] at (cqpptwo.south) (cqpptwobox) {$\left\{ -\delta \partial \vartheta^{(q - 1)} = (-1)^q \cdot \partial\omega^{(q)} \right\}$};
\node [below=8pt,fill=white,draw,dashed] at (cqmponeppone.south) (cqmponepponebox) {$\left\{ \beta^{(q - 1)} - (-1)^{q - 1} \cdot \vartheta^{(q - 1)} \right\}$};

\draw[complex arrow,|->,shorten <=3pt,shorten >=3pt] (c0nbox.east) -- (c1nbox.west); 
\draw[complex arrow,|->,shorten <=3pt,shorten >=3pt] (c1nmonebox.east) -- (c2nmonebox.west); 
\draw[complex arrow,|->,shorten <=3pt,shorten >=3pt] (cqmponepptwobox.east) -- (cqpptwobox.west); 

\node [minimum height=0.35in, minimum width=1.4in] (cqponeppone) at (\leftend+\truncation*\horizontal+3*\horizontal, \topend-\truncation*\vertical-2*\vertical) {};
\node [minimum height=0.35in, minimum width=1.4in] (cqponep) at (\leftend+\truncation*\horizontal+3*\horizontal, \topend-\truncation*\vertical-3*\vertical) {};
\node [minimum height=0.35in, minimum width=1.4in] (cqp) at (\leftend+\truncation*\horizontal+2*\horizontal, \topend-\truncation*\vertical-3*\vertical) {};

\definecolor{mycolor}{rgb}{0.2, 1, 0.4}
    \begin{pgfonlayer}{background}
        \fill[opacity=0.1,blue,rounded corners=5pt,draw=black,draw opacity=0.6,dashed,thick] (c0n.south west) -- (c0n.north west) -- (c0n.north east) -- (cqponep.north east) -- (cqponep.south east) -- (cqponep.south west) -- cycle;
        \fill[opacity=0.1,mycolor,rounded corners=5pt,draw=black,draw opacity=0.6,dashed,thick] (c0nmone.south west) -- (c0n.south west) -- (cqponep.south west) -- (cqp.south west) -- cycle;
        \fill[opacity=0.1,cyan,rounded corners=5pt,draw=black,draw opacity=0.6,dashed,thick] (c0n.north east) -- (c1n.north east) -- (cqponeppone.north east) -- (cqponep.north east) -- cycle;
    \end{pgfonlayer}

  \end{tikzpicture}%
};
    \else
      \node[anchor=center] at (current page.center) {\rotatebox{-90}{%
  \begin{tikzpicture}[x=1in,y=1in]
    
  \end{tikzpicture}%
}};
    \fi
  \end{tikzpicture}%
  \clearpage
}
 
\newif\ifhtmlconversion
\iflatexml \htmlconversiontrue \fi
\ifcsname lxSVG@insertpicture\endcsname \htmlconversiontrue \fi
 
\ifhtmlconversion
  \let\insertfullpagefiguredummy\fullpagefigureforhtml
\else
  \let\insertfullpagefiguredummy\fullpagefigureforpdf
\fi

\begin{document}
\maketitle

\begin{abstract}
To each ternary sextic $f(X_0, X_1, X_2)$ whose associated plane curve is smooth, the \textit{Shioda construction} attaches a smooth sextic fourfold $X \subset \mathbb{P}^5$ whose defining equation $f(X_0, X_1, X_2) - f(Y_0, Y_1, Y_2)$ is fixed under the involution $\iota : (X_0, X_1, X_2, Y_0, Y_1, Y_2) \mapsto i \cdot (Y_0, Y_1, Y_2, -X_0, -X_1, -X_2)$. The induced action $\iota^* : H^4(X, \mathbb{Q}) \to H^4(X, \mathbb{Q})$ 
fixes a Hodge substructure $H \subset H^4(X, \mathbb{Q})$ whose \textit{Hodge coniveau} is 1. By the general Hodge conjecture, we expect that there should exist a divisor $Y \subset X$ for which $H \subset \ker\left( H^4(X, \mathbb{Q}) \to H^4(X \setminus Y, \mathbb{Q}) \right)$. We verify this prediction in case the Waring rank of $f(X_0, X_1, X_2)$ takes on its minimum possible value, partially answering a question of Voisin (J.\ Math.\ Sci.\ Univ.\ Tokyo '15).
\end{abstract}

\section{Introduction}

For each smooth sextic fourfold $X \subset \mathbb{P}^5$ whose defining equation $F(X_0, X_1, X_2, Y_0, Y_1, Y_2)$ is invariant under the projective linear involution $\iota : (X_0, X_1, X_2, Y_0, Y_1, Y_2) \mapsto i \cdot (Y_0, Y_1, Y_2, -X_0, -X_1, -X_2)$, the induced action $\iota^* : H^4(X, \mathbb{C}) \to H^4(X, \mathbb{C})$ acts as $-1$ on $H^{4, 0}(X)$. The Hodge substructure $H^4(X, \mathbb{Q})^+ \subset H^4(X, \mathbb{Q})$ consisting of \textit{invariant} classes is thus of \textit{Hodge coniveau} 1, and so is expected to be of \textit{geometric coniveau} 1.

Voisin \cite{Voisin:2015aa} isolates precisely this ``particularly interesting'' special case of the generalized Hodge conjecture for the sake of its relationship to the generalized Bloch conjecture. Indeed, for $f(X_0, X_1, X_2)$ defining a smooth plane sextic $C \subset \mathbb{P}^2$, the smooth fourfold $X \subset \mathbb{P}^5$ defined by $f(X_0, X_1, X_2) - f(Y_0, Y_1, Y_2)$ is invariant under the involution $\iota$ (see also Shioda and Katsura \cite[Rem.~1.10]{Shioda:1979aa}). If the invariant Hodge substructure $H^4(X, \mathbb{Q})^+ \subset H^4(X, \mathbb{Q})$ were \textit{parameterized by 1-cycles} in the sense of \cite[Def.~0.3]{Voisin:2015aa}---this notion, due to Vial \cite{Vial:2013aa}, strengthens somewhat that of \textit{geometric coniveau 1}---then so too would Voisin's \cite[Ques.~3.4]{Voisin:2015aa} be true; this question in turn captures a special case of the generalized Bloch conjecture, pertaining to the self-product $S \times S$ of a smooth K3 surface $S$ that's a double cover of $\mathbb{P}^2$ ramified over $C$.

Though the notions ``parameterized by 1-cycles'' and ``geometric coniveau 1'' are equivalent for \textit{general} sextic hypersurfaces $X \subset \mathbb{P}^5$ (as Voisin notes \cite[Rem.~0.5]{Voisin:2015aa}), their equivalence for \textit{all} smooth sextics apparently remains unknown, though it would be implied by the Lefschetz standard conjecture \cite{Voisin:2015aa}.

In any case, in this work, we unconditionally confirm the prediction of the generalized Hodge conjecture as regards $H^4(X, \mathbb{Q})^+ \subset H^4(X, \mathbb{Q})$, for a particularly simple family of sextics $F(X_0, X_1, X_2, Y_0, Y_1, Y_2) = f(X_0, X_1, X_2) - f(Y_0, Y_1, Y_2)$. The \textit{Waring rank} of a ternary sextic $f(X_0, X_1, X_2)$ is the minimal integer $s$ for which $f(X_0, X_1, X_2)$ admits an expression of the shape
\begin{equation*}f(X_0, X_1, X_2) = L_1(X_0, X_1, X_2)^6 + \cdots + L_s(X_0, X_1, X_2)^6,\end{equation*}
for $L_1(X_0, X_1, X_2), \ldots , L_s(X_0, X_1, X_2)$ linear forms. (This notion can of course be defined for arbitrary numbers of variables and arbitrary degrees.) This notion is classical and quite difficult, and is surveyed at length in Iarrobino and Kanev \cite{Iarrobino:1999aa}. If the curve $C \subset \mathbb{P}^2$ defined by $f(X_0, X_1, X_2)$ is \textit{smooth}, then the Waring rank of $f(X_0, X_1, X_2)$ is trivially \textit{at least} 3. On the other hand, for dimension reasons, no fewer than 10 forms can express a \textit{general} ternary sextic, and in fact 10 forms suffice \cite[Cor.~1.62]{Iarrobino:1999aa}. The worst-case number of forms is larger---at least 12 (see De Paris \cite{De-Paris:2016aa})---and the exact maximum rank is not known.

The techniques of this work serve to settle the above case of the generalized Hodge conjecture in the Waring-rank-3 case. This case represents a projective 8-dimensional family of smooth curves $C \subset \mathbb{P}^2$ within the 27-dimensional set of all such curves. In fact, our techniques are enough to handle, at the very least, 17 projective dimensions within the 235-dimensional space of all \textit{invariant} sextics $X \subset \mathbb{P}^5$; that they might be pushed further seems plausible. (Only 8 of our 17 dimensions fall under the Shioda--Katsura construction.)

We now survey our approach.
In the first part of the paper, we reduce the Hodge-theoretic statement at issue to a purely algebraic one, essentially by closely studying Griffiths's \cite{Griffiths:1969aa} classical description of the primitive cohomology of hypersurfaces. We fix for now an arbitrary smooth hypersurface $X \subset \mathbb{P}^n$ of degree $d$ and a pair $p + q = n - 1$. Voisin \cite[Thm.~6.1]{Voisin:2003aa} asserts a variant of Griffiths's theorem whereby the Poincaré residues $\operatorname{Res}_X (\omega_A)$, where $A$ is homogeneous of degree $d \cdot (q + 1) - (n + 1)$ and $\omega_A \coloneqq \frac{A \cdot \Omega}{F^{q + 1}}$ is its \textit{Griffiths form}, exhaust the Hodge level $F^p H^{n - 1}(X, \mathbb{C})_{\text{prim}}$. 
It follows from \cite[Thm.~6.1]{Voisin:2003aa} (see also \cite[Thm.~6.5]{Voisin:2003aa}) that the class in $H^{n - 1}(X, \mathbb{C})$ associated with $A$ vanishes on $X \cap U$---where $U \subset \mathbb{P}^n$, say, is an arbitrary principal open whose intersection with $X$ is nonempty---if and only if there exists a holomorphic differential $n - 1$-form $\beta \in \Gamma(U \setminus X, \Omega_{\mathbb{P}^n}^{n-1})$ such that $\left. \omega_A\right|_{U} - \partial \beta$ extends holomorphically across $X$.

In Section \ref{basic} below, we work out a new, elementary proof of this result (see Theorem \ref{main}). Our proof boasts the advantage over Voisin's \cite[Thm.~6.5]{Voisin:2003aa} whereby it follows purely formally from, say, Lewis's \cite[Lec.~9]{Lewis:1999aa} elementary rendition of Griffiths's method; in particular, it doesn't depend on \textit{Deligne's theorem} to the effect that the spectral sequence associated to the Hodge filtration on the logarithmic complex $\Omega^\bullet_{\mathbb{P}^n}\left( \log X \right)$ degenerates at $E_1$ (see Remark \ref{remark} for further discussion of this). Our proof further clarifies the matter of $\beta$'s pole order along $X$. In fact, it's at most $q$; that is, 
$\beta$ can be taken from $\Gamma\left(U, \Omega^{n - 1}_{\mathbb{P}^n}(qX) \right)$, if it can be taken at all. We show as much \textit{without} invoking Deligne's theorem or the Grothendieck comparison theorem; upon admitting the latter, we further deduce that $\beta$ can be assumed rational (see Theorem \ref{grothendieck}). Our approach is Čech-cohomological and algorithmic in flavor, and essentially enriches that of Lewis \cite[Lec.~9]{Lewis:1999aa}.

Using these results, and specializing $n = 5$ and $d = 6$, we reëxpress the prediction issued by the generalized Hodge conjecture as regards $H^4(X, \mathbb{Q})^+ \subset H^4(X, \mathbb{Q})$ as the following \textit{equivalent}---and purely algebraic---condition. That is, for $F(X_0, X_1, X_2, Y_0, Y_1, Y_2)$ $\iota$-invariant, our question becomes whether, for each $q \in \{0, \ldots , 4\}$ and each $\iota$-\textit{anti-invariant} polynomial $A \in R^{6 \cdot q}$, a rational vector field $v = \sum_{i = 0}^5 v_i \cdot \frac{\partial}{\partial Z_i}$---whose components are homogeneous rational functions of degree $6 \cdot q - 5$, with denominators not contained in $(F)$---can be chosen in such a way that
\begin{equation}\label{key}F^{q + 1} \mid A - q \cdot dF(v) + F \cdot \operatorname{div}(v)\end{equation}
holds in the local ring $\mathbb{C}[X_0, X_1, X_2, Y_0, Y_1, Y_2]_{(F)}$; here, $\operatorname{div}(v)$ denotes $v$'s \textit{divergence}. 
In fact, as we argue below (see Theorem \ref{containment}), the sole case $q = 2$ is alone both necessary and sufficient.
Variants of this expression appear, at least implicitly, in Griffiths \cite[(4.5)]{Griffiths:1969aa}, Lewis \cite[9.32.~Thm.]{Lewis:1999aa} and Voisin \cite[Lem.~6.11]{Voisin:2003aa}.

In Section \ref{solve}, we turn our attention to the algebraic partial differential equation \eqref{key}. We first apply the linear change of variables $(X_0, X_1, X_2, Y_0, Y_1, Y_2) \mapsto (X_0, X_1, X_2, i \cdot Y_0, i \cdot Y_1, i \cdot Y_2) \eqqcolon (U_0, U_1, U_2, V_0, V_1, V_2)$, under which $\iota$ identifies with the \textit{block swap} involution $\sigma : (U_0, U_1, U_2, V_0, V_1, V_2) \mapsto (V_0, V_1, V_2, U_0, U_1, U_2)$ and $f(X_0, X_1, X_2) - f(Y_0, Y_1, Y_2)$ with $f(U_0, U_1, U_2) + f(V_0, V_1, V_2)$. In this way, we reduce our task to solving, for $F(U_0, U_1, U_2, V_0, V_1, V_2)$ of that latter shape, the equation \eqref{key} for $q = 2$ and for each $\sigma$-anti-invariant $A \in R^{12}$. The prototypical such $F$ is the Fermat sextic $F(U_0, U_1, U_2, V_0, V_1, V_2) = U_0^6 + U_1^6 + U_2^6 + V_0^6 + V_1^6 + V_2^6$, which we handle first. To this end, we develop an explicit algorithm (see Construction \ref{full_construction}) that constructively solves \eqref{key} in case $F$ is the Fermat sextic, $q = 2$ and $A \in R^{12}$ is $\sigma$-anti-invariant. Our algorithm extends Bostan, Lairez and Salvy's \textit{Griffiths--Dwork reduction} \cite[Alg.~1]{Bostan:2013aa} by explicitly solving---using a Fermat-specific method---the successive Jacobian remainders that arise during that algorithm, as opposed to merely accumulating them out-of-band (see also Remark \ref{bls}). Indeed, Griffiths--Dwork reduction serves only to decide the exactness of $\omega_A$ on all of $\mathbb{P}^n \setminus X$, a condition that obtains only if $A \in R^{6 \cdot q}$ yields the zero cohomology class in $H^{n - 1}(X, \mathbb{C})$ under Griffiths's identification. We must rather decide, for each \textit{nontrivial} anti-invariant $A$, whether there is a proper principal open $U \subset \mathbb{P}^n$ meeting $X$ on which $\left. \omega_A \right|_{U}$ becomes exact (up to holomorphic error). 
Producing this principal open---essentially, an algebraic cycle---is far harder. 

We now sketch our Fermat-specific method. 
Using Griffiths--Dwork reduction, we reduce the problem at hand to that of treating just those monomials $A = Z^a$ of total degree $6 \cdot q$, where $q \in \{1, 2\}$, each of whose exponents is at most 4. We exploit the purely combinatorial fact whereby, for each such $a = (a_0, \ldots, a_5)$---i.e., for which each $a_i \in \{0, \ldots , 4\}$ and $\sum_{i = 0}^5 a_i = 6 \cdot q$---there exists a nonempty and proper subset $I \subset \{0, \ldots , 5\}$ such that $\sum_{i \in I} a_i + |I| = 6$. (This fact fails as soon as $q > 2$; see Remark \ref{fails}.) This fact can be proven by brute enumeration (see Lemma \ref{combinatorial} below). Each such $I$ yields a ``partial Euler'' vector field $\nu_I \coloneqq \sum_{i \in I} Z_i \cdot \frac{\partial}{\partial Z_i}$ for which, first of all, $dF(\nu_I) \notin (F)$ holds (see Lemma \ref{lie} below), and for which, moreover, the witness $v \coloneqq \frac{A \cdot \nu_I}{q \cdot dF(\nu_I)}$ directly solves \eqref{key}, as an elementary quotient-rule calculation demonstrates (see Lemma \ref{quotient}). We use $A$'s anti-invariance to assert the vanishing of its fully reduced counterpart in the base case of our algorithm (see Lemma \ref{final}). This completes our treatment of the Fermat sextic $F$.

Finally, for each sextic $f(U_0, U_1, U_2) = L_0(U_0, U_1, U_2)^6 + L_1(U_0, U_1, U_2)^6 + L_2(U_0, U_1, U_2)^6$ of Waring rank 3 whose associated plane curve $C \subset \mathbb{P}^2$ is smooth, the (necessarily linearly independent) forms $L_0$, $L_1$ and $L_2$ define an element of $\text{GL}_3(\mathbb{C})$. Upon applying that element blockwise, we obtain a $\sigma$-equivariant change of coordinates that identifies the Fermat sextic with $f(U_0, U_1, U_2) + f(V_0, V_1, V_2)$, and under which our Fermat solution too transports. This solves \eqref{key} for each such sextic $f(U_0, U_1, U_2) + f(V_0, V_1, V_2)$ (see Theorem \ref{last}).

Interestingly, our divisors $W \subset \mathbb{P}^5$---i.e., defined by the equations $dF(\nu_I)$, for $I \subset \{0, \ldots , 5\}$ nonempty and proper---can be interpreted geometrically, as we now explain. We fix such an $I \subset \{0, \ldots , 5\}$. For notational convenience, we assume that $I = \{i, \ldots, 5\}$ for some $i \in \{1, \ldots , 5\}$. 
The prescription $\varphi : (Z_0 : \cdots : Z_5) \mapsto \left( (Z_0 : \cdots : Z_{i - 1}), (Z_i : \cdots : Z_5) \right)$ defines a rational map $\varphi : \mathbb{P}^5 \dashrightarrow \mathbb{P}^{i - 1} \times \mathbb{P}^{5 - i}$, undefined along the union of the two projective linear subspaces $Z_0 = \cdots = Z_{i - 1} = 0$ and $Z_i = \cdots = Z_5 = 0$. Over each point $\left( (z_0 : \cdots : z_{i - 1}), (z_i : \cdots : z_5) \right)$, the fiber of $\varphi$ is $\set*{ \left( \lambda \cdot z_0 : \cdots : \lambda \cdot z_{i - 1} : \mu \cdot z_i : \cdots : \mu \cdot z_5\right) ; (\lambda : \mu) \in \mathbb{P}^1 }$ (or more properly, that line minus its two endpoints). We claim that the divisor $W \subset \mathbb{P}^5$ cut out by $dF(\nu_I)$ is such that $W \cap X$ is the ramification divisor of $X$ under this fibration. To see this, note that the map $(a_1, \ldots , a_{i - 1}, b_{i + 1}, \ldots , b_5, t) \mapsto (1 : a_1 : \cdots : a_{i - 1} : t : t \cdot b_{i + 1} : \cdots : t \cdot b_5)$ maps $\mathbb{C}^4 \times \mathbb{C}^*$ isomorphically 
onto the open chart $\set{ (Z_0 : \cdots : Z_5) \in \mathbb{P}^5 ; Z_0 \neq 0 \text{ and } Z_i \neq 0 }$ of $\mathbb{P}^5$. Moreover, each $(a_1, \ldots , a_{i - 1}, b_{i + 1}, \ldots , b_5)$ defines a fiber of $\varphi$, and the intersection of $X$ with that fiber is the vanishing locus of the univariate sextic $f_{a, b}(t) \coloneqq F(1, a_1, \ldots , a_{i - 1}, t, t \cdot b_{i + 1}, \ldots , t \cdot b_5)$. The ramification locus of $X$ is defined by the vanishing conditions $f_{a, b}(t) = 0$ and $f'_{a, b}(t) = 0$. By the chain rule, $t \cdot f'_{a, b}(t)$ is exactly the pullback of $\sum_{i \in I} Z_i \cdot \frac{\partial F}{\partial Z_i}$.

We discuss a few further aspects of the special case of the generalized Hodge conjecture treated by this work that make it especially attractive. Typical cases of the generalized Hodge conjecture present the difficulty whereby no purely algebraic criterion---i.e., other than that furnished tautologically by the conjecture itself---serves to identify, within the set of all complex cohomology classes, those whose vanishing the conjecture predicts. 
In other words, it's hard to know which classes one must kill, if he is to establish the conjecture, in each given case.
While the Hodge filtration $F^p H^k(X, \mathbb{C}) = \operatorname{im} \left( \mathbb{H}^k(X, \Omega^{\geq p}_X) \to \mathbb{H}^k(X, \Omega^\bullet_X) \right)$ on $H^k(X, \mathbb{C})$
admits a purely algebraic---i.e., hypercohomological---description, the underlying Betti $\mathbb{Q}$-structure on that space doesn't;  in particular, one can't in general detect purely algebraically which de Rham classes lie in the complexification of the rational Hodge substructure at issue.

The problem treated by this work, on the other hand, bypasses this issue by \textit{fixing} a Hodge structure---known \textit{a priori} to be of Hodge coniveau 1---and then asking whether the generalized Hodge conjecture's prediction as regards it holds. Moreover, the Hodge structure at issue---or rather, its complexification---can be characterized purely algebraically (i.e., as that consisting of the involution-invariant algebraic de Rham hypercohomology classes). This fact bypasses, in our case, the problem of ``detecting Hodge classes'', and makes our variant of Griffiths's explicit description extremely effective. Our refinement of that description, together with the results of our Subsection \ref{algebraic}, serves to establish not just the sufficiency of the algebraic partial differential equation \eqref{key} (which seems more-or-less standard), but further its necessity. 

Moreover, the target \eqref{key} of this work's reductions can be attacked computationally, as we now explain. For each $\sigma$-invariant sextic $F$ and each $\sigma$-anti-invariant homogeneous polynomial $A$ of the right degree, each \textit{candidate} denominator $D$---homogeneous, as well as involution-invariant, say---serves to induce a finite-dimensional space of solutions $v = (v_0, \ldots , v_5)$, itself consisting of those vector fields of the shape $v = \left( \frac{A_0}{D}, \ldots , \frac{A_5}{D} \right)$, for numerators $A_i$ again homogeneous and of the right degree (we explain this reformulation further in Subsection \ref{algebraic}). The equation \eqref{key} yields a linear system on the coefficients of these numerators, whose consistency can be checked in finite time (at least in principle). 
The problem, then, is to find denominators $D$ that cause ``clever'' cancellations to take place in \eqref{key}. (This latter problem, of course, algebraically distills the very cycle-theoretic statement at issue.) While we've been able to achieve just this for $F$ the Fermat sextic, and for a family of sextics closely related to that one, doing as much for general $F$ seems difficult. This work is the first of which we're aware that manages to reduce a case of the generalized Hodge conjecture to a further problem that's equivalent \textit{and} is amenable to computer-aided search.

Indeed, it is not for lack of effort that we have failed---thus far---to solve \eqref{key} for, say, a nonempty Zariski-open subset's worth of $\sigma$-invariant sextics $F(U_0, U_1, U_2, V_0, V_1, V_2)$ (this latter desideratum captures the full ``challenging'' question asked by Voisin \cite{Voisin:2015aa}). Our inability---despite its concreteness and simplicity---to solve that equation for a broader class of sextics $F$ reflects, presumably, the great difficulty with which algebraic cycles are inevitably produced. We suggest the equation \eqref{key} as a point of departure for future efforts 
(see also Question \ref{question} below).

\subsection{Background and Notation}

We adopt the notation of Lewis \cite[Lec.~9]{Lewis:1999aa}, Görtz and Wedhorn \cite{Gortz:2023aa} and Peters and Steenbrink \cite{Peters:2008aa}.
We fix an integer $n > 1$ and a smooth hypersurface $X \subset \mathbb{P}^n$ whose defining equation $F(Z_0, \ldots , Z_n) \in \mathbb{C}[Z_0, \ldots , Z_n]$ is of degree $d$, say. We write $\mathscr{U} = \left\{ U_j \right\}_{j \in J}$ for an open cover of $\mathbb{P}^n$.

We make use of Čech hypercohomology (see e.g.\ \cite[Def.~Rem.~21.74]{Gortz:2023aa}). That is, we have the total complex $\check{C}(\mathscr{U}, \Omega^\bullet_{\mathbb{P}^n}(\log X)) \coloneqq \operatorname{Tot}\left((\check{C}^i(\mathscr{U}, \Omega^j_{\mathbb{P}^n}(\log X)))_{i, j} \right)$ on $\mathbb{P}^n$. 
By convention, $i \in \{0, \ldots , n\}$ indexes $(\check{C}^i(\mathscr{U}, \Omega^j_{\mathbb{P}^n}(\log X)))_{i, j}$'s columns and $j \in \{0, \ldots , n\}$ its rows (see \cite[(F.17.1)]{Gortz:2023aa}). We recall that the total complex's differential flips each of that double complex's odd-indexed vertical differentials (this is \cite[Def.~F.87]{Gortz:2023aa}). 
Similarly, we have the total complex
$\check{C}(\mathscr{U} \cap X, \Omega^\bullet_X) \coloneqq \operatorname{Tot}\left((\check{C}^i(\mathscr{U} \cap X, \Omega^j_X))_{i, j} \right)$ on $X$.

We use the notational convention whereby $\delta$ signifies the Čech differential, $\partial$ signifies the holomorphic differential, and $d$ signifies the total differential (i.e., that of the total complex). Following \cite[(F.14)]{Gortz:2023aa}, we write $\pi_n(C)$ for the $n$\textsuperscript{th} component of a complex $C$, as well as $Z^n(C)$, $B^n(C)$ and $H^n(C)$ respectively for its $n$\textsuperscript{th} cocycle, coboundary and cohomology objects.

We may freely assume that each component of $\mathscr{U} = \left\{ U_j \right\}_{j \in J}$ is affine. Since $\mathbb{P}^n$ is separated, each iterated intersection $U_{j_0, \ldots , j_k} \coloneqq U_{j_0} \cap \cdots \cap U_{j_k}$---and further, each set of the form $U_{j_0, \ldots , j_k} \cap U$, where $U \subset \mathbb{P}^n$ is itself affine---is affine; its analytification is thus Stein. Since $X \subset \mathbb{P}^n$ is closed, each set of the form $U_{j_0, \ldots , j_k} \cap X$ and $U_{j_0, \ldots , j_k} \cap U \cap X$ is again affine, with Stein analytification. Each algebraic sheaf used in this work is coherent, so that its analytification is a coherent analytic sheaf. Cartan's Theorem B (see e.g.\ \cite[Ex.~B.17]{Peters:2008aa}) thus implies that on each of the open sets catalogued above, each such sheaf's higher cohomology vanishes. We use this result in two ways. First, it guarantees that the covers $\mathscr{U}$, $\mathscr{U} \cap X$, $\mathscr{U} \cap U$ and $\mathscr{U} \cap U \cap X$ are all Leray for all coherent analytic sheaves at hand. In particular, we obtain, for example, the guarantees $\check{H}^q\left( \mathscr{U} \cap X, \Omega^p_X \right) \cong H^q\left( X, \Omega^p_X \right)$ (this is the usual Leray theorem; see e.g.\ \cite[Thm.~B.16]{Peters:2008aa}). We further obtain global hypercohomological results like $\check{H}^{n - 1}\left( \mathscr{U} \cap X, \Omega^\bullet_X \right) \cong \mathbb{H}^{n - 1}\left( X, \Omega^\bullet_X \right)$ (see e.g.\ \cite[Cor.~21.82]{Gortz:2023aa}). We note that, by the holomorphic Poincaré lemma, this latter set in turn satisfies $\mathbb{H}^{n - 1}(X, \Omega^\bullet_X) \cong H^{n - 1}(X, \mathbb{C})$ (see e.g.\ \cite[Thm.~B.18]{Peters:2008aa}).
In addition, since the analytification of $U$ itself is Stein, we separately use Cartan's theorem to guarantee that each coherent analytic sheaf's higher cohomology vanishes on $U$. 
Combining these results, we conclude in particular that each positive-indexed Čech cohomology group of each coherent analytic sheaf on $\mathbb{P}^n$, taken with respect to the cover $\mathscr{U} \cap U$ of $U \subset \mathbb{P}^n$, vanishes.
Finally, it holds that each exact sequence of coherent analytic $\mathcal{O}_{\mathbb{P}^n}$-modules used below specializes, over each of the open sets introduced above, to an exact sequence on sections. This fact justifies those among our steps below that ``lift'' Čech cochains.

Similarly, for each affine Zariski-open $U \subset \mathbb{P}^n$, Serre's cohomological criterion for affineness (see e.g.\ Hartshorne \cite[III~Thm.~3.7]{Hartshorne:1977aa}) implies that each exact sequence of quasicoherent algebraic sheaves 
specializes over $U$ to an exact sequence on algebraic sections.


We fix a pair $p + q = n - 1$. For each $q \geq 1$, we have the long exact sequence \cite[9.29.~Lem.]{Lewis:1999aa}, namely
\begin{equation*}0 \to \Omega^{p + 1}_{\mathbb{P}^n}(\log X) \hookrightarrow \Omega_{\mathbb{P}^n}^{p + 1}(X) \xlongrightarrow{\partial_1} \Omega_{\mathbb{P}^n}^{p + 2}(2X) \mathbin{/} \Omega_{\mathbb{P}^n}^{p + 2}(X) \xlongrightarrow{\partial_2} \cdots \xlongrightarrow{\partial_q} \Omega_{\mathbb{P}^n}^n((q + 1)X) \mathbin{/} \Omega_{\mathbb{P}^n}^n(qX) \to 0.\end{equation*}
In what follows, in order to economically consolidate notation, we call the leftmost and rightmost sheaves just above $\ker \partial_1$ and $\ker \partial_{q + 1}$. Moreover, as in \cite[\nopp 9.30]{Lewis:1999aa}, we break up the sequence above into short exact sequences
\begin{equation}\label{short}0 \to \ker \partial_\ell \to \mathcal{Q}_\ell \to \ker \partial_{\ell + 1} \to 0\end{equation}
for $\ell \in \{1, \ldots , q\}$; here, we use the notational device whereby we define $\mathcal{Q}_\ell \coloneqq \Omega_{\mathbb{P}^n}^{p + \ell}(\ell X) \mathbin{/} \Omega_{\mathbb{P}^n}^{p + \ell}((\ell - 1)X)$ if $\ell > 1$ and $\mathcal{Q}_\ell \coloneqq \Omega^{p + 1}_{\mathbb{P}^n}(X)$ if $\ell = 1$. These cases correspond respectively to \cite[(9.4)]{Lewis:1999aa} and \cite[(9.5)]{Lewis:1999aa}.

These exact sequences induce boundary maps $\check{H}^0\left( \mathscr{U}, \ker \partial_{q + 1}\right) \to \check{H}^1 \left( \mathscr{U}, \ker \partial_q \right) \to \cdots \to \check{H}^q\left(\mathscr{U}, \ker \partial_1 \right)$. It is a consequence of Bott's vanishing theorem that the first among these maps is a surjection, while the rest are isomorphisms (see the proof of \cite[9.31.~Prop.]{Lewis:1999aa}).

Following Lewis (see \cite[9.16.~Prop.]{Lewis:1999aa}), we use angular brackets to denote the \textit{contraction action} between a vector field and a differential form.

\paragraph{Acknowledgements.} I would like to thank Raju Krishnamoorthy and Robert Laterveer for various very helpful discussions.

\section{Effective Griffiths--Lewis} \label{basic}

In this section, we develop an effective variant of Griffiths's characterization \cite{Griffiths:1969aa} of the primitive cohomology of hypersurfaces. We fix an integer $n > 1$, a smooth hypersurface $X \subset \mathbb{P}^n$ of degree $d$, and a pair of nonnegative integers $p$ and $q$ for which $p + q = n - 1$. Our goal is to make effectively decidable, for each class in $H^{n - 1}(X, \mathbb{C})_{\text{prim}}$, the matter of that class's vanishing outside of some closed algebraic subset $Y \subset X$, itself given as $Y = W \cap X$, say, for some hypersurface $W \subset \mathbb{P}^n$. We write $U \coloneqq \mathbb{P}^n \setminus W$.

Using the degeneration at $E_1$ of the spectral sequence associated to the Hodge filtration on the logarithmic complex $\Omega^\bullet_{\mathbb{P}^n}(\log X)$---itself a deep result, due to Deligne (see \cite[Thm.~8.35~(b)]{Voisin:2002aa})---Voisin \cite[Thm.~6.5]{Voisin:2003aa} proves a modern form of Griffiths's result, whereby the respective Poincaré residues of the differential $n$-forms $\omega_A = \frac{A \cdot \Omega}{F^{q + 1}}$, for $A$ homogeneous of degree $d \cdot (q + 1) - (n + 1)$, exhaust  $F^p H^{n - 1}(X, \mathbb{C})_{\text{prim}}$. Here, $\Omega \coloneqq \langle \nu, \operatorname{vol} \rangle$ is the contraction of the volume form on $\mathbb{C}^{n + 1}$ by the Euler vector field $\nu \coloneqq \sum_{i = 0}^n Z_i \cdot \frac{\partial}{\partial Z_i}$. 
Voisin's variant implies that the cohomology class corresponding to the residue of $\omega_A$ dies on $X \cap U$ if and only if there exists a holomorphic $n - 1$-form $\beta \in \Gamma\left( U \setminus X, \Omega^{n - 1}_{\mathbb{P}^n} \right)$ such that $\left. \omega_A \right|_U - \partial \beta$ extends holomorphically across $X$. \textit{A priori}, the pole order of $\beta$ along $X$ could be arbitrarily large. In fact, $\beta$ could have essential singularities both along $X$ and along the boundary of $U$, a pathology that the Grothendieck comparison theorem apparently serves only partially to rule out (we discuss this phenomenon below).

In this section, we clarify the pole order of $\beta$ along $X$. In fact, we show (see Theorem \ref{main} below) that the class of $\operatorname{Res}_X(\omega_A)$ vanishes on $X \cap U$ if and only if there exists a $\beta \in \Gamma\left( U \setminus X, \Omega^{n - 1}_{\mathbb{P}^n} \right)$ such that $\left. \omega_A \right|_U - \partial \beta$ is holomorphic and \textit{$\beta$'s pole order along $X$ is at most $q$}; that is, we can freely assume $\beta \in \Gamma\left( U, \Omega^{n - 1}_{\mathbb{P}^n}(qX) \right)$. Moreover, we prove as much without assuming Deligne's theorem; in fact, our technique implies a very weak form of that theorem (see Remark \ref{weak_form}). In Theorem \ref{grothendieck}, we show that $\beta$ can moreover be assumed algebraic. 

We now sketch our approach, which adapts Lewis's treatment \cite[Lec.~9]{Lewis:1999aa}. Lewis shows that $\text{Prim}^{p, q}(X)$, viewed as a subspace of $H^q(X, \Omega^p_X)$ (see \cite[(9.2)]{Lewis:1999aa}), is exhausted by a certain multistep procedure which, beginning with the homogeneous polynomial $A \in R^{d \cdot (q + 1) - (n + 1)}$, first constructs the $n$-form $\omega_A \coloneqq \frac{A \cdot \Omega}{F^{q + 1}}$ on $\mathbb{P}^n$, before applying to $\omega_A$ a sequence of connecting homomorphisms (see \cite[9.31.~Prop.]{Lewis:1999aa}), and finally taking the result's residue.

Lewis's approach, while simple and elementary, yields just the graded pieces $\text{Prim}^{p, q}(X)$ of $H^{n - 1}(X, \mathbb{C})_{\text{prim}}$ with respect to the Hodge filtration. We begin, in Subsection \ref{total}, by producing Čech hypercohomology classes that respectively lift these graded pieces (see Lemma \ref{inductive_construction}). We produce these total Čech cocycles by the aid of a sort of iterative correction; that is, beginning with the cocycle $(\omega_A, 0, \ldots , 0)$, we iteratively produce offsets---themselves obtained by pole-order reduction---that serve to reduce that cocycle's worst-case pole order, while keeping intact its closedness (see Construction \ref{construction}). Adapting Lewis's \cite[9.28.~Prop.]{Lewis:1999aa} and \cite[9.31.~Prop.]{Lewis:1999aa}, we prove that the respective residues of the cocycles we obtain in this way, together with ambient cohomology, exhaust $H^{n - 1}(X, \mathbb{C})$ (see Theorem \ref{containment}). 

In Subsection \ref{criterion}, having thus explicitly described the elements of $H^{n - 1}(X, \mathbb{C})$, we turn to these classes' vanishing. Our Theorem \ref{main} asserts that the total cocycle $s = (s^{(0)}, \ldots , s^{(q)}, 0, \ldots , 0)$ associated to $A$ becomes almost-exact on $U$ (i.e., up to holomorphic error) if and only if its Griffiths form $\omega_A = \frac{A \cdot \Omega}{F^{q + 1}}$ becomes almost-exact on $U$ with respect to a primitive $\beta \in \Gamma\left( U \setminus X, \Omega^{n - 1}_{\mathbb{P}^n} \right)$ whose pole order along $X$, crucially, is at most $q$; that is, we may take $\beta \in \Gamma\left( U, \Omega^{n - 1}_{\mathbb{P}^n}(qX) \right)$. Our proof of this fact is again elementary and Čech-cohomological, albeit rather intricate. To achieve the forward direction (see Construction \ref{sufficiency} and Lemma \ref{correctness_primitive}), we build, using the assumed primitive $\beta$, an almost-total primitive of $s$ in the log complex, again by iteratively reducing pole order. In the reverse direction (see Construction \ref{necessity} and Lemma \ref{necessity_correctness}), we show that each among the components of a given \textit{total} primitive can be expressed as the sum of a $\delta$-closed form and a $\partial$-closed form; in this endeavor, we repeatedly apply Cartan's \textit{Theorem B}.  

In Subsection \ref{algebraic}, we discuss the algebraicity of $\omega_A$'s primitive. To show that the algebraic partial differential equation \eqref{key} (see also \eqref{thing} below) is not just \textit{sufficient} for the vanishing of $\operatorname{Res}_X(\omega_A)$ on $X \cap U$ but further \textit{necessary}, we must show that the primitive furnished by Theorem \ref{main} can be assumed algebraic. Our Theorem \ref{grothendieck} below achieves just this, making use of a certain kind of \textit{algebraic} pole-order reduction, this time built with the aid of a key formula of Griffiths (see Lemma \ref{griffiths}). In that theorem, the assumed finite pole order of $\beta$ near $X$ proves apparently essential; that is, given merely a primitive $\beta \in \Gamma\left( U \setminus X, \Omega^{n - 1}_{\mathbb{P}^n} \right)$ whose singularities near $X$ could in principle be essential, we're not sure how to produce an algebraic primitive $\eta$ analogous to that of our Theorem \ref{grothendieck}. This fact suggests that our sharpening of \cite[Thm.~6.5]{Voisin:2003aa} (i.e., that achieved by our Theorem \ref{main}) is critical, at least as far as the necessity of \eqref{key} is concerned.

\subsection{Producing Total Cocycles} \label{total}

Lewis's approach, given $A \in R^{d \cdot (q + 1) - (n + 1)}$, first defines the form $\omega_A \coloneqq \frac{A \cdot \Omega}{F^{q + 1}}$. It applies to $\omega_A$ a sequence of connecting homomorphisms, so producing a sequence of elements $\omega_A^{(q + 1 - \ell)} \in H^{q - \ell + 1}\left( \mathbb{P}^n, \ker \partial_\ell \right)$, 
say, for $\ell \in \{q + 1, \ldots , 1\}$ (i.e., working downwards). At the end, Lewis's procedure applies to $\omega_A^{(q)} \in H^q\left( \mathbb{P}^n, \Omega^{p + 1}_{\mathbb{P}^n}(\log X) \right)$ the map $\operatorname{Res}_X$ induced by the Poincaré residue $\Omega^{p + 1}_{\mathbb{P}^n}(\log X) \to \Omega^p_X$, and so obtains an element of $\text{Prim}^{p, q}(X) \subset H^q(X, \Omega^p_X)$.

In this subsection, we introduce a lifting technique that mimics Lewis's, but more closely keeps track of the intermediate data that that technique produces, and so outputs a \textit{total} Čech cocycle, as opposed to a graded element. We sketch our construction's basic idea. We begin with the total cocycle consisting \textit{just} of $\omega_A$ in its top-left cell. That total cocycle is \textit{not} valued in the logarithmic complex $\check{C}(\mathscr{U}, \Omega^\bullet_{\mathbb{P}^n}(\log X))$, unless $q = 0$ ($\omega_A$'s pole order along $X$ is $q + 1$). For $q$ iterations, we ``correct'' the total cocycle so far constructed by adding to it a totally-closed offset. That offset causes the standing cocycle's bottom-right-most item's poles to become logarithmic, at the cost of spawning a \textit{new} bottom-right-most item, whose pole order, on the other hand, is one lower than the standing cocycle's bottom-right item's had been. After repeating this procedure $q$ times, we obtain an honest logarithmic total cocycle. 
The $(p + 1, q)$-indexed component of the final cocycle we obtain in this way coincides, as a cohomology class in $H^q\left( \mathbb{P}^n, \Omega^{p + 1}_{\mathbb{P}^n}(\log X) \right)$, with Lewis's element $\omega_A^{(q)}$ above (up to a possible sign discrepancy).

In order to keep this section typographically manageable, we introduce a few notational devices. We suppress the braces, as well as the lower indices, attached to the various Čech cochains we use. In particular, we write each Čech cochain as a single symbol, together with a parenthesized superscript that indicates that cochain's arity. For example, our $\omega^{(q)}$ below is a Čech $q$-cochain, an element of $\check{C}^{q}\left( \mathscr{U}, \Omega_{\mathbb{P}^n}^{p + 1}(\log X) \right)$.

We begin with our main construction.

\begin{construction}[Construction of Total Cocycle] \label{construction}
\textit{Input.} A smooth hypersurface $X \subset \mathbb{P}^n$, a pair $p + q = n - 1$, and a homogeneous $A \in R^{d \cdot (q + 1) - (n + 1)}$. \\
\textit{Output.} A total cocycle $s = (s^{(0)}, \ldots , s^{(q)}, 0, \ldots , 0)$ in $Z^n \left( \check{C}(\mathscr{U}, \Omega^\bullet_{\mathbb{P}^n}(\log X)) \right)$. 
\begin{itemize}
\item Set $\omega_A \coloneqq \frac{A \cdot \Omega}{F^{q + 1}}$ and initialize the Čech 0-cochain $\omega^{(0)} \coloneqq \left\{ \left. \omega_A\right|_{U_{j}} \right\}_{j}$ in $\check{C}^0\left( \mathscr{U}, \Omega_{\mathbb{P}^n}^n((q + 1)X) \right)$.
\item \textbf{for} $\ell \in \{q, \ldots , 1\}$ \textbf{do}:
\begin{itemize}
\item Initialize $\alpha^{(q - \ell)}_{\ell + 1} \coloneqq 0$.
\item \textbf{for} $h \in \{\ell, \ldots , 1\}$ \textbf{do}:
\begin{itemize}
\item By induction on $\ell$ (see below), $\partial\omega^{(q - \ell)} \in \check{C}^{q - \ell}\left( \mathscr{U}, \Omega^{p + \ell + 2}_{\mathbb{P}^n}(\log X) \right)$ 
has logarithmic poles (and so \textit{a fortiori} order-$h + 1$ poles). Moreover, by induction on $h$, $\omega^{(q - \ell)} - \partial\alpha^{(q - \ell)}_{h + 1} \in \check{C}^{q - \ell}\left( \mathscr{U}, \Omega_{\mathbb{P}^n}^{p + \ell + 1}((h + 1)X) \right)$ has order-$h + 1$ poles at most.
Thus, by the exact sequence
\begin{equation*}\Omega_{\mathbb{P}^n}^{p + \ell}(hX) \xlongrightarrow{\partial} \Omega_{\mathbb{P}^n}^{p + \ell + 1}((h + 1)X) \mathbin{/} \Omega_{\mathbb{P}^n}^{p + \ell + 1}(hX) \xlongrightarrow{\partial} \Omega_{\mathbb{P}^n}^{p + \ell + 2}((h + 2)X) \mathbin{/} \Omega_{\mathbb{P}^n}^{p + \ell + 2}((h + 1)X),\end{equation*}
we can lift $\omega^{(q - \ell)} - \partial\alpha^{(q - \ell)}_{h + 1}$ to an element $\gamma^{(q - \ell)}_h \in \check{C}^{q - \ell}\left( \mathscr{U}, \Omega_{\mathbb{P}^n}^{p + \ell}\left( hX \right) \right)$ say,
such that $\omega^{(q - \ell)} - \partial\alpha^{(q - \ell)}_{h + 1} - \partial\gamma^{(q - \ell)}_h$ has poles of order $h$ at most. Update $\alpha^{(q - \ell)}_h \coloneqq \alpha^{(q - \ell)}_{h + 1} + \gamma^{(q - \ell)}_h$.
\end{itemize}
\item Define $\vartheta^{(q - \ell)} \coloneqq \alpha^{(q - \ell)}_1$ in $\check{C}^{q - \ell}\left( \mathscr{U}, \Omega_{\mathbb{P}^n}^{p + \ell}(\ell X) \right)$.
\item Further, define $\omega^{(q - \ell + 1)} \coloneqq (-1)^{q - \ell + 1} \cdot \delta \vartheta^{(q - \ell)} $ in $\check{C}^{q - \ell + 1}\left( \mathscr{U}, \Omega_{\mathbb{P}^n}^{p + \ell}(\ell X) \right)$.
\end{itemize}
\item Output the total cochain $\big( \omega^{(0)} - \partial \vartheta^{(0)} , \ldots , \omega^{(q - 1)} - \partial \vartheta^{(q - 1)} , \omega^{(q)}, 0, \ldots , 0 \big)$.
\end{itemize}
\end{construction}

The following claim establishes the correctness of Construction \ref{construction}.
\begin{lemma} \label{inductive_construction}
The total cochain $s = (s^{(0)}, \ldots , s^{(q)}, 0, \ldots , 0)$ output by Construction \ref{construction} takes values in the logarithmic complex $\check{C}(\mathscr{U}, \Omega^\bullet_{\mathbb{P}^n}(\log X))$ and is a total cocycle. The class of $s^{(q)}$ in $\check{H}^q\left( \mathscr{U}, \Omega_{\mathbb{P}^n}^{p + 1}(\log X) \right)$ is $(-1)^{\binom{q + 1}{2}}$ times the image of $\omega^{(0)}$ under the chain of boundary maps $\check{H}^0\left( \mathscr{U}, \ker \partial_{q + 1}\right) \to \cdots \to \check{H}^q\left(\mathscr{U}, \ker \partial_1 \right)$.
\end{lemma}
\begin{proof}
We prove by induction that for each $\ell \in \{q + 1, \ldots , 1\}$, 
the Čech cochains $\omega^{(0)}, \ldots , \omega^{(q - \ell + 1)}$ and $\vartheta^{(0)}, \ldots , \vartheta^{(q - \ell)}$ constructed by Construction \ref{construction} satisfy the following invariants. 
\begin{enumerate}
\item\label{closed} $\omega^{(q - \ell + 1)} \in \check{C}^{q - \ell + 1}\left( \mathscr{U}, \Omega_{\mathbb{P}^n}^{p + \ell}(\ell X)\right)$ is $\delta$-closed.
\item\label{important} $\partial\omega^{(q - \ell + 1)} \in \check{C}^{q - \ell + 1}\left( \mathscr{U}, \Omega_{\mathbb{P}^n}^{p + \ell + 1}(\log X) \right)$ has logarithmic poles.
\item\label{chain} $\omega^{(q - \ell + 1)}$ defines a class in $\check{H}^{q - \ell + 1}\left( \mathscr{U}, \ker \partial_\ell \right)$, which in fact represents $(-1)^{\binom{q - \ell + 2}{2}}$ times the image of $\omega^{(0)}$ under the chain of boundary maps $\check{H}^0\left( \mathscr{U}, \ker \partial_{q + 1}\right) \to \cdots \to \check{H}^{q - \ell + 1}\left(\mathscr{U}, \ker \partial_\ell \right)$.
\item\label{closedness} The total cochain $\left( \omega^{(0)} - \partial\vartheta^{(0)}, \ldots , \omega^{(q - \ell)} - \partial\vartheta^{(q - \ell)}, \omega^{(q - \ell + 1)}, 0, \ldots , 0 \right)$ is a total cocycle. Its first $q - \ell + 1$ elements have logarithmic poles, whereas $\omega^{(q - \ell + 1)}$ has order-$\ell$ poles.
\end{enumerate}
We note that \ref{important} is exploited in an essential way \textit{during} the pole-lifting procedure above.

These items granted, the conclusion of the lemma essentially is immediate. Indeed, \ref{closedness} implies that the elements $\omega^{(0)} - \partial\vartheta^{(0)}, \ldots , \omega^{(q - 1)} - \partial\vartheta^{(q - 1)}$ have logarithmic poles, while $\omega^{(q)}$'s are at worst simple; on the other hand, \ref{important} implies that $\partial\omega^{(q)}$'s poles are logarithmic, so that $\omega^{(q)}$'s too in fact are (recall \cite[9.21.~Prop.-Def.]{Lewis:1999aa}) and the entire cocycle is log-valued.

In the base case $\ell = q + 1$, the inductive invariants above trivially hold by inspection. We fix an index $\ell \in \{q, \ldots , 1\}$; i.e., we proceed by \textit{downward} induction (i.e., in the same order as Construction \ref{construction}). Item \ref{closed} above is immediate, since $\omega^{(q - \ell + 1)} \coloneqq (-1)^{q - \ell + 1} \cdot \delta \vartheta^{(q - \ell)}$ is in fact a Čech coboundary.

The output $\vartheta^{(q - \ell)}$ of Construction \ref{construction}'s inner loop is defined in just such a way that $\omega^{(q - \ell)} - \partial\vartheta^{(q - \ell)} \in \check{C}^{q - \ell}\left( \mathscr{U}, \Omega_{\mathbb{P}^n}^{p + \ell + 1}(X) \right)$ has simple poles. Moreover, that quantity's holomorphic differential is $\partial\omega^{(q - \ell)}$, which has logarithmic poles (by the inductive case of \ref{important}). We conclude that $\omega^{(q - \ell)} - \partial\vartheta^{(q - \ell)}$ has logarithmic poles. Since $\omega^{(q - \ell + 1)}$'s pole order is clearly at most $\ell$, this observation proves the last sentence of item \ref{closedness}. To prove that item's first sentence, we note that the total cochain $\left( 0, \ldots , 0, -\partial \vartheta^{(q - \ell)}, \omega^{(q - \ell + 1)} , 0, \ldots , 0 \right)$ is closed, since $\delta \left(- \partial\vartheta^{(q - \ell)}\right) + (-1)^{q - \ell + 1} \cdot \partial \omega^{(q - \ell + 1)} = -\delta \partial\vartheta^{(q - \ell)} + \partial\delta \vartheta^{(q - \ell)} = 0$; taken together with the inductive case of \ref{closedness}, this observation proves \ref{closedness}.

We note that $\partial\omega^{(q - \ell + 1)} = (-1)^{q - \ell + 1} \cdot \partial\delta \vartheta^{(q - \ell)} = (-1)^{q - \ell + 1} \cdot -\delta\left( \omega^{(q - \ell)} - \partial\vartheta^{(q - \ell)} \right)$, where the last equality follows from the Čech-closedness of $\omega^{(q - \ell)}$ (itself guaranteed by the inductive case of \ref{closed}). Since, as we just proved, $\omega^{(q - \ell)} - \partial\vartheta^{(q - \ell)}$ has logarithmic poles, we see that $\partial\omega^{(q - \ell + 1)}$ too does, so that \ref{important} is proved.

Since $\omega^{(q - \ell)} - \partial\vartheta^{(q - \ell)}$ has logarithmic poles---and so, \textit{a fortiori}, poles of order at most $\ell$---it represents the zero cochain in $\check{C}^{q - \ell}\left( \mathscr{U}, \ker \partial_{\ell + 1} \right)$ (recall \eqref{short}). In other words, $\partial\vartheta^{(q - \ell)}$ and $\omega^{(q - \ell)}$ represent the same element in $\check{C}^{q - \ell}\left( \mathscr{U}, \ker \partial_{\ell + 1} \right)$. 
Since $\omega^{(q - \ell)}$ is itself Čech-closed (again by induction), it follows that the Čech differential $\delta \vartheta^{(q - \ell)}$ maps, under $\partial$, to the zero element in $\check{C}^{q - \ell + 1}\left( \mathscr{U}, \ker \partial_{\ell + 1} \right)$, so that $\delta \vartheta^{(q - \ell)}$ itself takes values in $\check{C}^{q - \ell + 1}\left( \mathscr{U}, \ker \partial_\ell \right)$ (recall again \eqref{short}), and in fact represents the image of the class of $\omega^{(q - \ell)}$ under the connecting map $\check{H}^{q - \ell}\left( \mathscr{U}, \ker \partial_{\ell + 1} \right) \to \check{H}^{q - \ell + 1} \left( \mathscr{U}, \ker \partial_{\ell} \right)$ attached to \eqref{short}. The element $\omega^{(q - \ell + 1)} \coloneqq (-1)^{q - \ell + 1} \cdot \delta \vartheta^{(q - \ell)}$ thus represents $(-1)^{q - \ell + 1}$ times this coboundary, a fact that proves \ref{chain} by induction.
\end{proof}



In the following theorem, we show that for each $p + q = n - 1$, the respective residues of the cocycles output by Construction \ref{construction} span, alongside $X$'s ambient cohomology, the Hodge level $F^p H^{n - 1}(X, \mathbb{C})$.
The main ingredient of our proof is already present within Lewis \cite[9.31.~Prop.]{Lewis:1999aa}. We lift that componentwise surjectivity statement into one pertaining to total classes; we must also deal with ambient cohomology.

We again make use of the total complex $\check{C}(\mathscr{U} \cap X, \Omega^\bullet_X)$ attached to the holomorphic Čech double complex on $X$. 
We recall that $H^{n - 1}(X, \mathbb{C}) \cong \mathbb{H}^{n - 1}(X, \Omega^\bullet_X)$, and again represent that latter space concretely as $\check{H}^{n - 1}\left( \mathscr{U} \cap X, \Omega^\bullet_X \right)$.
We write $j : X \hookrightarrow \mathbb{P}^n$ for the inclusion. 
For each $q \in \{0, \ldots , n - 1\}$, we write $\mathcal{R}_q \coloneqq \operatorname{Span}_\mathbb{C} \set*{ [\operatorname{Res}_X(s)] }_A$, i.e., the span of the cohomology classes of the respective residues of \textit{all} cocycles output by Construction \ref{construction}, as $A \in R^{d \cdot (q + 1) - (n + 1)}$ varies (i.e., including all possible choices of lift made during that algorithm). By the first sentence of Lemma \ref{inductive_construction}, these residues are indeed closed, so taking their cohomology classes makes sense. 
Finally, we write $\mathcal{A} \coloneqq \operatorname{im}\left( j^* : \check{H}^{n - 1} \left( \mathscr{U}, \Omega^\bullet_{\mathbb{P}^n} \right) \to \check{H}^{n - 1} \left( \mathscr{U} \cap X, \Omega^\bullet_X \right) \right)$. 

We recall the row-based filtration $_{II}F$ on $\check{C}(\mathscr{U} \cap X, \Omega^\bullet_X)$ (see \cite[(F.25.1)]{Gortz:2023aa}) and its associated spectral sequence. By \cite[(F.25.2)]{Gortz:2023aa}, we have the first page $_{II}E_1^{p, q} = \check{H}^q\left( \mathscr{U} \cap X, \Omega^p_X \right)$. In order to make the below proof simpler, we admit the degeneration at $E_1$ of the standard Hodge--de Rham spectral sequence on $X$ (though cf.\ Remark \ref{e1}). %

\begin{theorem} \label{containment}
For each pair $p + q = n - 1$, $_{II}F^p \check{H}^{n - 1}\left( \mathscr{U} \cap X, \Omega^\bullet_X \right) = \left( \mathcal{A} \cap {}_{II}F^p \check{H}^{n - 1}\left( \mathscr{U} \cap X, \Omega^\bullet_X \right) \right) + \mathcal{R}_q$.
\end{theorem}
\begin{proof}
We fix a pair $p + q = n - 1$. We abbreviate $\mathcal{S}_p \coloneqq \left( \mathcal{A} \cap {}_{II}F^p \check{H}^{n - 1}\left( \mathscr{U} \cap X, \Omega^\bullet_X \right) \right) + \mathcal{R}_q$. We note immediately that $\mathcal{S}_p \subset {}_{II}F^p \check{H}^{n - 1}\left( \mathscr{U} \cap X, \Omega^\bullet_X \right)$, since $\mathcal{R}_q \subset {}_{II}F^p \check{H}^{n - 1}\left( \mathscr{U} \cap X, \Omega^\bullet_X \right)$; this latter fact follows essentially by inspection of Construction \ref{construction}.

We argue that the restriction of the quotient map ${}_{II}F^p \check{H}^{n - 1}\left( \mathscr{U} \cap X, \Omega^\bullet_X \right) \to \operatorname{gr}^p_{_{II}F} \check{H}^{n - 1}\left( \mathscr{U} \cap X, \Omega^\bullet_X \right)$ to $\mathcal{S}_p$ is surjective. To do this, we let $t \in {}_{II}F^p \check{H}^{n - 1}\left( \mathscr{U} \cap X, \Omega^\bullet_X \right)$ be arbitrary, and write $[t] \in \operatorname{gr}^p_{_{II}F} \check{H}^{n - 1}\left( \mathscr{U} \cap X, \Omega^\bullet_X \right) \cong \check{H}^q \left( \mathscr{U} \cap X, \Omega^p_X \right)$ for its graded image. Applying Lewis \cite[9.27.~Rem.]{Lewis:1999aa} and \cite[9.28.~Prop.]{Lewis:1999aa} to $[t]$, we obtain an expression $[t] = \operatorname{Res}_X(r) + j^*(a)$ in $\check{H}^q \left( \mathscr{U} \cap X, \Omega^p_X \right)$, for $r \in \check{H}^q \left( \mathscr{U}, \Omega_{\mathbb{P}^n}^{p + 1}(\log X) \right)$ 
and $a \in \check{H}^q\left( \mathscr{U}, \Omega^p_{\mathbb{P}^n} \right)$. We handle these summands one at a time. 

As for $\operatorname{Res}_X(r)$, by the argument of \cite[9.31.~Prop.]{Lewis:1999aa}, 
our Lemma \ref{inductive_construction}'s second sentence implies that, for some total cocycle $s = (s^{(0)}, \ldots , s^{(q)}, 0, \ldots , 0)$ output by Construction \ref{construction}, $\operatorname{Res}_X(s^{(q)})$'s class in $\check{H}^q \left( \mathscr{U} \cap X, \Omega^p_X \right)$ equals $\operatorname{Res}_X(r)$. Further, the total cohomology class of $\operatorname{Res}_X(s)$ resides in $\mathcal{R}_q$ by definition of that latter set. This implies that $\operatorname{Res}_X(r)$ is in the image of $\mathcal{S}_p \to \operatorname{gr}^p_{_{II}F} \check{H}^{n - 1}\left( \mathscr{U} \cap X, \Omega^\bullet_X \right)$.

As for $a$, if $p \neq q$, then $a = 0$ and there is nothing to show. Otherwise, \cite[9.23.~Lem.]{Lewis:1999aa} shows that we may concretely represent the hyperplane class $a \in \check{H}^p\left( \mathscr{U}, \Omega^p_{\mathbb{P}^n} \right)$ by the Čech cocycle $\alpha = \left( \frac{\delta \beta}{2\pi i d} \right)^{\smile p}$, where $\beta \in \check{C}^0\left( \mathscr{U}, \Omega^1_{\mathbb{P}^n}(\log X) \right)$ is a further cochain explicitly constructed in \cite[9.24.~Sub.]{Lewis:1999aa} (with holomorphic coboundary). The cochain $\alpha$ is of course $\delta$-closed; as it happens, one can check directly that it's moreover $\partial$-closed (since $\beta$ itself is), and so immediately represents a totally-closed cochain---say, $\tilde{\alpha}$---in $\check{C}^{n - 1}\left( \mathscr{U}, \Omega^\bullet_{\mathbb{P}^n} \right)$. The total cohomology class of $j^*(\tilde{\alpha})$ of course lives in $\mathcal{A} \cap {}_{II}F^p \check{H}^{n - 1}\left( \mathscr{U} \cap X, \Omega^\bullet_X \right)$ (only $\tilde{\alpha}$'s $(p, p)$\textsuperscript{th} component is nonzero), and so again belongs to $\mathcal{S}_p$; moreover, its image in $\operatorname{gr}^p_{_{II}F} \check{H}^{n - 1}\left( \mathscr{U} \cap X, \Omega^\bullet_X \right)$ is $j^*(a)$. We've thus shown that $[t]$ is in the image of $\mathcal{S}_p$, as required.

We now induct on the Hodge filtration. We fix an element $t \in {}_{II}F^p \check{H}^{n - 1}\left( \mathscr{U} \cap X, \Omega^\bullet_X \right)$ as above. For each $\bar{p} \in \{p, \ldots , n\}$, we consider the inductive condition whereby there exist elements $s_i \in \mathcal{S}_i$, for $i \in \{p, \ldots , \bar{p} - 1\}$, for which $t - \sum_{i = p}^{\bar{p} - 1} s_i \in {}_{II}F^{\bar{p}} \check{H}^{n - 1}\left( \mathscr{U} \cap X, \Omega^\bullet_X \right)$. In the base case $\bar{p} = p$, the condition is vacuous and there's nothing to show. We let $\bar{p} \in \{p, \ldots , n - 1\}$ be arbitrary and assume the inductive condition; we write $\bar{q} \coloneqq n - 1 - \bar{p}$. Repeating the argument just made, with $\bar{p}$ and $\bar{q}$ in place of $p$ and $q$, and $\bar{t} \coloneqq t - \sum_{i = p}^{\bar{p} - 1} s_i$ in place of $t$, we obtain an element $s_{\bar{p}} \in \mathcal{S}_{\bar{p}}$, say, such that $\bar{t} - s_{\bar{p}} = t - \sum_{i = p}^{\bar{p}} s_i \in {}_{II}F^{\bar{p} + 1} \check{H}^{n - 1}\left( \mathscr{U} \cap X, \Omega^\bullet_X \right)$. Carrying through the induction, we conclude that $t - \sum_{i = p}^{n - 1} s_i \in {}_{II}F^n \check{H}^{n - 1}\left( \mathscr{U} \cap X, \Omega^\bullet_X \right) = \{0\}$.

It remains to argue that, for each $\bar{p} \in \{p, \ldots , n - 1\}$, $s_{\bar{p}} \in \mathcal{S}_p$. For this, it suffices to show that $\mathcal{S}_{\bar{p}} \subset \mathcal{S}_p$; the content of this claim is that $\mathcal{R}_{\bar{q}} \subset \mathcal{R}_q$ holds for each $\bar{q} \in \{0, \ldots , q\}$.

To prove this latter claim, we fix an element $\bar{q} \in \{0, \ldots , q - 1\}$ (the case $\bar{q} = q$ is tautological) and an input $\bar{A} \in R^{d \cdot (\bar{q} + 1) - (n + 1)}$, and arbitrarily run Construction \ref{construction} on $\bar{A}$; we write that run's output as $\left( \bar{\omega}^{(0)} - \partial \bar{\vartheta}^{(0)} , \ldots , \bar{\omega}^{(\bar{q} - 1)} - \partial\bar{\vartheta}^{(\bar{q} - 1)} , \bar{\omega}^{(\bar{q})}, 0, \ldots , 0 \right)$. We set $A \coloneqq \bar{A} \cdot F^{q - \bar{q}}$, and run Construction \ref{construction} on $A$. We claim that the output of that large run can be made to match that of the arbitrary small run. For each outer loop index $\ell \in \{q, \ldots , q - \bar{q}\}$, we assume by induction that $\omega^{(q - \ell)} = \bar{\omega}^{(q - \ell)}$ holds; we note that the latter element's poles along $X$ are of order $\ell - q + \bar{q} + 1$ at most (recall condition \ref{closed} of the proof of Lemma \ref{inductive_construction}). The base case $\ell = q$ amounts to $\omega_A = \omega_{\bar{A}}$. For each outer loop index $\ell \in \{q, \ldots , q - \bar{q}\}$, assuming the inductive equality, we note that the prescription that sets $\gamma_h^{(q - \ell)} \coloneqq 0$ for the first $q - \bar{q}$ inner loop indices $h \in \{\ell, \ldots , \ell - q + \bar{q} + 1\}$ and replicates the small run's choices for each trailing index $h \in \{\ell - q + \bar{q}, \ldots , 1\}$ yields legal choices as far as Construction \ref{construction} is concerned (here, we use the inductive equality and the pole bound on $\bar{\omega}^{(q - \ell)}$). For each $\ell \in \{q, \ldots , q - \bar{q} + 1\}$, this prescription induces the equalities $\vartheta^{(q - \ell)} = \bar{\vartheta}^{(q - \ell)}$ and $\omega^{(q - \ell + 1)} = \bar{\omega}^{(q - \ell + 1)}$, and so preserves the inductive hypothesis. In the case $\ell = q - \bar{q}$, the prescription has the inner loop in fact do nothing at all (the second stretch above is empty), and so yields $\vartheta^{(\bar{q})} = 0$, so that $\omega^{(\bar{q})} - \partial \vartheta^{(\bar{q})} = \bar{\omega}^{(\bar{q})}$ and $\omega^{(\bar{q} + 1)} = 0$. The algorithm can continue to do nothing for each $\ell \in \{q - \bar{q} - 1, \ldots , 1\}$, so inducing the required identically zero trailing stretch $\left( \omega^{(\bar{q} + 1)} - \partial \vartheta^{(\bar{q} + 1)}, \ldots , \omega^{(q - 1)} - \partial \vartheta^{(q - 1)}, \omega^{(q)}\right)$.

We conclude that $t = \sum_{i = p}^{n - 1} s_i$ resides in $\mathcal{S}_p$, as required. This completes the proof of the theorem.
\end{proof}

\begin{remark} \label{e1}
Our proof above could have gotten away with refraining from using the Hodge--de Rham degeneration at $E_1$ on $X$. To do as much, it should have instead let $[t] \in \operatorname{gr}^p_{_{II}F} \check{H}^{n - 1}\left( \mathscr{U} \cap X, \Omega^\bullet_X \right) \cong {}_{II}Z_\infty^{p, q} \mathbin{/} {}_{II}B_\infty^{p, q}$, a group that's of course \textit{a priori} unequal to ${}_{II}E_1^{p, q} \cong \check{H}^q\left( \mathscr{U} \cap X, \Omega^p_X \right)$. Lifting $[t]$ arbitrarily to ${}_{II}Z_\infty^{p, q} \mathbin{/} {}_{II}B_1^{p, q}$ and using ${}_{II}Z_\infty^{p, q} \mathbin{/} {}_{II}B_1^{p, q} \subset {}_{II}Z_1^{p, q} \mathbin{/} {}_{II}B_1^{p, q}$, that proof could have applied \cite[9.27.~Rem.]{Lewis:1999aa} rather on that lift of $[t]$. The preimages $\operatorname{Res}_X(s)$ and $j^*(\tilde{\alpha})$ would then have hit said lift modulo ${}_{II}B_1^{p, q}$, and so \textit{a fortiori} modulo ${}_{II}B_\infty^{p, q}$, establishing in any case the required surjectivity onto $\operatorname{gr}^p_{_{II}F} \check{H}^{n - 1}\left( \mathscr{U} \cap X, \Omega^\bullet_X \right)$.
\end{remark}

\begin{remark} \label{remark}
More importantly, our proof above---and in fact, this entire section---refrains from invoking \textit{Deligne's theorem}, the deeper fact whereby the spectral sequence $_{II}E_1^{p + 1, q} = \check{H}^q\left( \mathscr{U}, \Omega^{p + 1}_{\mathbb{P}^n}(\log X) \right)$ itself degenerates at $E_1$ (see \cite[Thm.~8.35~(b)]{Voisin:2002aa}). In this respect, our work departs from Voisin's treatment \cite[Thm.~6.5]{Voisin:2003aa} of Griffiths's methods. This work's analytic inputs are fairly minimal; they amount just to the holomorphic Poincaré lemma (recall \cite[Thm.~B.18]{Peters:2008aa}), Cartan's Theorem B (see \cite[Ex.~B.17]{Peters:2008aa}) and the analytic datum furnished by \cite[9.24.~Sub.]{Lewis:1999aa}. We \textit{avoid} using not just Hodge's and Deligne's theorems but further the fact that $\mathbb{H}^n\left( \mathbb{P}^n, \Omega^\bullet_{\mathbb{P}^n}(\log X) \right) \cong H^n\left( \mathbb{P}^n \setminus X, \mathbb{C} \right)$, itself a consequence of the quasi-isomorphism of $\Omega^\bullet_{\mathbb{P}^n}(\log X) \to i_*\Omega^\bullet_{\mathbb{P}^n \setminus X}$
(see \cite[Thm.~4.2~1)]{Peters:2008aa} and \cite[Prop.~4.3]{Peters:2008aa}).

Our proof also bypasses the topological characterization of the residue map---developed in, say, Voisin \cite[\S~6.1.1]{Voisin:2003aa}---whereby the map $\mathbb{H}^n\left( \mathbb{P}^n, \Omega^\bullet_{\mathbb{P}^n}(\log X) \right) \to \mathbb{H}^{n - 1}\left( X, \Omega^\bullet_X \right)$ induced by $\operatorname{Res}_X : \Omega^{p + 1}_{\mathbb{P}^n}(\log X) \to j_*\Omega^p_X$ identifies with 
the map $H^n\left( \mathbb{P}^n \setminus X, \mathbb{C} \right) \to H^{n + 1}\left( \mathbb{P}^n, \mathbb{P}^n \setminus X, \mathbb{C} \right) \cong H^{n + 1}\left( T, \partial T, \mathbb{C} \right) \cong H^{n - 1}\left(X, \mathbb{C} \right)$. Here, $T$ is a tubular neighborhood of $X$ in $\mathbb{P}^n$ and the last isomorphism is the Thom isomorphism. We refer to \cite[\S~6.1.1]{Voisin:2003aa} for further discussion of this topological characterization.
\end{remark}

\begin{remark} \label{weak_form}
It can be shown that our Lemma \ref{inductive_construction} implies a weak form of Deligne's theorem whereby, in the spectral sequence $_{II}E_1^{p + 1, q} = \check{H}^q\left( \mathscr{U}, \Omega^{p + 1}_{\mathbb{P}^n}(\log X) \right)$, for each $p + q = n - 1$, the containments ${}_{II}Z_\infty^{p + 1, q} \subset \cdots \subset {}_{II}Z_1^{p + 1, q}$ are equalities (it implies something similar for the Hodge--de Rham spectral sequence on $X$). On the other hand, our proof doesn't imply that ${}_{II}B_1^{p + 1, q} = \cdots = {}_{II}B_\infty^{p + 1, q}$, whereas of course Deligne's theorem \cite[Thm.~8.35~(b)]{Voisin:2002aa} does; nor does our proof touch the spectral sequence's off-diagonal cells.
\end{remark}

\vfill
\ifhtmlconversion
  \begin{center}\input{easter-egg.tex}\end{center}
\else
  \noindent\input{easter-egg.tex}\par
\fi
\newpage

\subsection{The Vanishing Criterion} \label{criterion}

In the theorem below, we isolate a simple criterion that detects the exactness (up to holomorphic error) of a total cochain $s$ output by Construction \ref{construction} on a principal open set $U \subset \mathbb{P}^n$. We abbreviate $C_U \coloneqq \check{C}(\mathscr{U} \cap U, \Omega^\bullet_{\mathbb{P}^n}(\log X))$ for the total complex attached to the logarithmic complex on $U \subset \mathbb{P}^n$.
We recall the \textit{residue} exact sequence
\begin{equation*}0 \to \Omega^\ell_{\mathbb{P}^n} \to \Omega^\ell_{\mathbb{P}^n}(\log X) \xlongrightarrow{\operatorname{Res}} \Omega^{\ell - 1}_X\to 0,\end{equation*}
valid for each $\ell \in \{1, \ldots , n\}$ (see \cite[p.~131]{Lewis:1999aa}). 
Since the residue map annihilates exactly the holomorphic forms, and is surjective, $\left. \operatorname{Res}_X(s)\right|_U$ 
is exact in $\check{C}(\mathscr{U} \cap X \cap U, \Omega^\bullet_X)$ if and only if there exists a total primitive $\tilde{s} = (\tilde{s}^{(0)}, \ldots , \tilde{s}^{(n - 1)})$ in $\pi_{n - 1}(C_U)$ such that $\left. s \right|_U - d\tilde{s}$ is componentwise-holomorphic (i.e., takes values in $\pi_n \left( \check{C}(\mathscr{U} \cap U, \Omega^\bullet_{\mathbb{P}^n}) \right)$).

For typographical convenience, we decline below to typographically indicate the fact of cochains' having been restricted to $U$; that is, we identify the cochains $\omega^{(0)}, \ldots , \omega^{(q)}$ and $\vartheta^{(0)}, \ldots , \vartheta^{(q - 1)}$ with their respective restrictions to $U \subset \mathbb{P}^n$ throughout this subsection. 

We fix $X \subset \mathbb{P}^n$, $p + q = n - 1$, and $A \in R^{d \cdot (q + 1) - (n + 1)}$, and write $s = (s^{(0)}, \ldots , s^{(q)}, 0, \ldots , 0)$ for the corresponding total cocycle output by Construction \ref{construction}. 
For the entirety of this subsection, we moreover assume $q \in \{1, \ldots , n - 1\}$. While the results of this subsection remain true when $q = 0$, that case must be proven differently, and in any case is irrelevant to our application.

\begin{theorem} \label{main}
There exists a cochain $\tilde{s} = \left( \tilde{s}^{(0)}, \ldots , \tilde{s}^{(n - 1)} \right)$ in $\pi_{n - 1}(C_U)$
such that $s|_U - d\tilde{s}$ is everywhere-holomorphic 
if and only if there exists an $n - 1$-form $\beta \in \Gamma\left( U, \Omega^{n - 1}_{\mathbb{P}^n}(qX) \right)$ such that $\left. \omega_A \right|_U - \partial \beta \in \Gamma\left( U, \Omega^n_{\mathbb{P}^n} \right)$.
\end{theorem}
\begin{proof}
We take the theorem's two implications one at a time.

\paragraph{Sufficiency.} 

The sufficiency of the condition of the theorem is established by the following construction. 

\begin{construction}[Construction of Total Primitive] \label{sufficiency}
\textit{Input.} An element $\beta \in \Gamma\left( U, \Omega^{n - 1}_{\mathbb{P}^n}(qX) \right)$ such that $\left. \omega_A \right|_U - \partial \beta \in \Gamma\left( U, \Omega^n_{\mathbb{P}^n} \right)$. \\ 
\textit{Output.} A total cochain $\tilde{s} = (\tilde{s}^{(0)}, \ldots , \tilde{s}^{(q - 1)}, 0, \ldots , 0)$ in $\pi_{n - 1}(C_U)$ such that $\left. s \right|_U - d\tilde{s}$ is holomorphic.
\begin{itemize}
\item Initialize the Čech 0-cycle $\beta^{(0)} \coloneqq \left\{ \left. \beta \right|_{U_{j} \cap U} \right\}_{j}$ in $\check{C}^0\left( \mathscr{U} \cap U, \Omega_{\mathbb{P}^n}^{n - 1}(qX) \right)$.
\item \textbf{for} $\ell \in \{q - 1, \ldots , 1\}$ \textbf{do}:
\begin{itemize}
\item Initialize $\alpha^{(q - \ell - 1)}_{\ell + 1} \coloneqq 0$.
\item \textbf{for} $h \in \{\ell, \ldots , 1\}$ \textbf{do}:
\begin{itemize}
\item By induction on $\ell$ (see below), $\partial\beta^{(q - \ell - 1)} - (-1)^{q - \ell - 1} \cdot \partial \vartheta^{(q - \ell - 1)}$ has logarithmic poles along $X$. 
Moreover, by induction on $h$, $\beta^{(q - \ell - 1)} - (-1)^{q - \ell - 1} \cdot \vartheta^{(q - \ell - 1)} - \partial\alpha^{(q - \ell - 1)}_{h + 1} \in \check{C}^{q - \ell - 1}\left( \mathscr{U} \cap U, \Omega_{\mathbb{P}^n}^{p + \ell + 1}\left( (h + 1)X \right) \right)$ has order-$h + 1$ poles at most.
Thus, again using
\begin{equation*}\Omega_{\mathbb{P}^n}^{p + \ell}(hX) \xlongrightarrow{\partial} \Omega_{\mathbb{P}^n}^{p + \ell + 1}((h + 1)X) \mathbin{/} \Omega_{\mathbb{P}^n}^{p + \ell + 1}(hX) \xlongrightarrow{\partial} \Omega_{\mathbb{P}^n}^{p + \ell + 2}((h + 2)X) \mathbin{/} \Omega_{\mathbb{P}^n}^{p + \ell + 2}((h + 1)X),\end{equation*}
lift $\beta^{(q - \ell - 1)} - (-1)^{q - \ell - 1} \cdot \vartheta^{(q - \ell - 1)} - \partial\alpha^{(q - \ell - 1)}_{h + 1}$ to $\gamma^{(q - \ell - 1)}_h \in \check{C}^{q - \ell - 1}\left( \mathscr{U} \cap U, \Omega_{\mathbb{P}^n}^{p + \ell}\left( hX \right) \right)$,
say,
such that $\beta^{(q - \ell - 1)} - (-1)^{q - \ell - 1} \cdot \vartheta^{(q - \ell - 1)} - \partial\alpha^{(q - \ell - 1)}_{h + 1} - \partial\gamma^{(q - \ell - 1)}_h$ has order-$h$ poles. Update $\alpha^{(q - \ell - 1)}_h \coloneqq \alpha^{(q - \ell - 1)}_{h + 1} + \gamma^{(q - \ell - 1)}_h$.
\end{itemize}
\item Define $\zeta^{(q - \ell - 1)} \coloneqq \alpha^{(q - \ell - 1)}_1$ in $\check{C}^{q - \ell - 1}\left( \mathscr{U} \cap U, \Omega_{\mathbb{P}^n}^{p + \ell}(\ell X) \right)$.
\item Further, define $\beta^{(q - \ell)} \coloneqq (-1)^{q - \ell} \cdot \delta \zeta^{(q - \ell - 1)}$ in $\check{C}^{q - \ell}\left( \mathscr{U} \cap U, \Omega_{\mathbb{P}^n}^{p + \ell}(\ell X) \right)$.
\end{itemize}
\item Output $\bigl( \beta^{(0)} - \vartheta^{(0)} - \partial\zeta^{(0)}, \ldots , \beta^{(q - 2)} - (-1)^{q - 2} \cdot \vartheta^{(q - 2)} - \partial\zeta^{(q - 2)}, \beta^{(q - 1)} - (-1)^{q - 1} \cdot \vartheta^{(q - 1)}, 0, \ldots , 0 \bigr)$. 
\end{itemize}
\end{construction}


\begin{lemma} \label{correctness_primitive}
The total cochain $\tilde{s} = (\tilde{s}^{(0)}, \ldots , \tilde{s}^{(q - 1)}, 0, \ldots , 0)$ output by Construction \ref{sufficiency} takes values in the log complex $\check{C}(\mathscr{U} \cap U, \Omega^\bullet_{\mathbb{P}^n}(\log X))$, and $s|_U - d\tilde{s}$ is everywhere-holomorphic (i.e., valued in $\check{C}(\mathscr{U} \cap U, \Omega^\bullet_{\mathbb{P}^n})$).
\end{lemma}
\begin{proof}
We again prove this result inductively. For each $\ell \in \{q, \ldots , 1\}$, we prove that the following inductive invariants hold.
\begin{enumerate}
\item\label{closed_primitive} $\beta^{(q - \ell)} \in \check{C}^{q - \ell}\left( \mathscr{U} \cap U, \Omega^{p + \ell}_{\mathbb{P}^n}(\ell X) \right)$ is Čech-closed.
\item\label{important_primitive} $\partial\beta^{(q - \ell)} - (-1)^{q - \ell} \cdot \partial\vartheta^{(q - \ell)} \in \check{C}^{q - \ell}\left( \mathscr{U} \cap U, \Omega^{p + \ell + 1}_{\mathbb{P}^n}(\log X) \right)$ has logarithmic poles.
\item\label{closedness_primitive} The cochain $\bigl( \beta^{(0)} - \vartheta^{(0)} - \partial\zeta^{(0)}, \ldots , \beta^{(q - \ell - 1)} - (-1)^{q - \ell - 1} \cdot \vartheta^{(q - \ell - 1)} - \partial\zeta^{(q - \ell - 1)}, \beta^{(q - \ell)} - (-1)^{q - \ell} \cdot \vartheta^{(q - \ell)}, -(-1)^{q - \ell + 1} \cdot \vartheta^{(q - \ell + 1)}, \ldots, -(-1)^{q - 1} \cdot \vartheta^{(q - 1)}, 0, \ldots , 0 \bigr)$ 
 is a total primitive of $s$, up to holomorphic error in the top term. Its first $q - \ell$ components have logarithmic poles; its next $\ell$ items have poles of order $(\ell, \ell - 1, \ldots , 1)$.
\end{enumerate}
Granting these conditions and carrying through the induction, we immediately obtain the statement of the lemma. Indeed, while the condition \ref{closedness_primitive} \textit{prima facie} implies merely that $ \beta^{(q - 1)} - (-1)^{q - 1} \cdot \vartheta^{(q - 1)}$'s poles are simple, the further condition \ref{important_primitive} implies that that quantity's differential's poles are logarithmic, so that $\beta^{(q - 1)} - (-1)^{q - 1} \cdot \vartheta^{(q - 1)}$'s poles are in fact logarithmic.

In the base case $\ell = q$, all of the inductive conditions above hold trivially except for the last. As for that one, we must compare $\left( \omega^{(0)} - \partial\vartheta^{(0)}, \ldots , \omega^{(q - 1)} - \partial\vartheta^{(q - 1)}, \omega^{(q)}, 0, \ldots , 0 \right)$ (or rather its restriction to $U$) to the total differential of $\left( \beta^{(0)} - \vartheta^{(0)}, -(-1)^1 \cdot \vartheta^{(1)}, \ldots, -(-1)^{q - 1} \cdot \vartheta^{(q - 1)}, 0, \ldots, 0 \right)$. At each positive-indexed component, these cochains agree. This is an explicit calculation, which we skip.
As for the top-left component, we have the discrepancy
\begin{equation*}\left( \omega^{(0)} - \partial\vartheta^{(0)} \right) - \partial\left( \beta^{(0)} - \vartheta^{(0)} \right) = \left\{ \left. \left( \left. \omega_A\right|_U - \partial \beta \right) \right|_{U_i} \right\}_i,\end{equation*}
which is holomorphic by our hypothesis on $\beta$. Finally, the respective pole orders of the components $\left( \beta^{(0)} - \vartheta^{(0)}, -(-1)^1 \cdot \vartheta^{(1)}, \ldots, -(-1)^{q - 1} \cdot \vartheta^{(q - 1)} \right)$ of the initial cochain are (at worst) $(q, q - 1, \ldots , 1)$, essentially by Construction \ref{construction} and by our hypothesis on $\beta$.

To prove the inductive case, we fix $\ell \in \{q - 1, \ldots , 1\}$. 
The Čech-closedness of $\beta^{(q - \ell)} \coloneqq (-1)^{q - \ell} \cdot \delta \zeta^{(q - \ell - 1)}$ in $\check{C}^{q - \ell}\left( \mathscr{U} \cap U, \Omega_{\mathbb{P}^n}^{p + \ell}(\ell X) \right)$ is immediate.

Construction \ref{sufficiency} is designed just in such a way that $\beta^{(q - \ell - 1)} - (-1)^{q - \ell - 1} \cdot \vartheta^{(q - \ell - 1)} - \partial\zeta^{(q - \ell - 1)}$ has simple poles. Moreover, that quantity's differential is $\partial\beta^{(q - \ell - 1)} - (-1)^{q - \ell - 1}\cdot \partial\vartheta^{(q - \ell - 1)}$, whose poles are logarithmic by induction (see item \ref{important_primitive}). We conclude that $\beta^{(q - \ell - 1)} - (-1)^{q - \ell - 1} \cdot \vartheta^{(q - \ell - 1)} - \partial\zeta^{(q - \ell - 1)}$ has logarithmic poles. This establishes the inductive case of the second part of \ref{closedness_primitive}. As for the first, we note that the offset $\left( 0, \ldots , -\partial\zeta^{(q - \ell - 1)}, \beta^{(q - \ell)}, 0, \ldots , 0 \right)$ is again totally closed by construction; adding it to the total cochain so far constructed thus doesn't change that cochain's status as a total primitive of $s$ (up to holomorphic error).

To prove that Construction \ref{sufficiency}'s main loop preserves the inductive condition \ref{important_primitive}, we note that
\begin{align*}
(-1)^{q - \ell} \cdot \partial\beta^{(q - \ell)} - \partial\vartheta^{(q - \ell)} &= \partial\delta \zeta^{(q - \ell - 1)} - \partial\vartheta^{(q - \ell)} \\
&= \omega^{(q - \ell)} - \partial \vartheta^{(q - \ell)} - \omega^{(q - \ell)} + \delta \partial \zeta^{(q - \ell - 1)} \\
&= \omega^{(q - \ell)} - \partial \vartheta^{(q - \ell)} - (-1)^{q - \ell} \cdot \delta \vartheta^{(q - \ell - 1)} + \delta \partial \zeta^{(q - \ell - 1)}  \\
&= \underbrace{\left( \omega^{(q - \ell)} - \partial \vartheta^{(q - \ell)} \right)}_{\text{logarithmic poles}} \mathrel{-} \delta \underbrace{\left( \beta^{(q - \ell - 1)} - (-1)^{q - \ell - 1} \cdot \vartheta^{(q - \ell - 1)} - \partial\zeta^{(q - \ell - 1)} \right)}_{\text{logarithmic poles}},
\end{align*}
which proves \ref{important_primitive}, since $\partial\beta^{(q - \ell)} - (-1)^{q - \ell} \cdot \partial\vartheta^{(q - \ell)}$ and $(-1)^{q - \ell} \cdot \partial\beta^{(q - \ell)} - \partial\vartheta^{(q - \ell)}$ agree up to sign. That the left-hand summand just above's poles are logarithmic is an immediate consequence of the construction of $s$ (i.e., of the fact that that total cocycle is valued in the log complex $\check{C}\left( \mathscr{U}, \Omega^\bullet_{\mathbb{P}^n}(\log X) \right)$); that the right-hand summand's poles are logarithmic was just proven. In the last step above, our \textit{ex nihilo} introduction of the term $\delta \beta^{(q - \ell - 1)} = 0$ is justified on the basis of the inductive condition \ref{closed_primitive}. This calculation proves \ref{important_primitive} and completes the proof of the lemma.
\end{proof}

Lemma \ref{correctness_primitive} completes the proof of sufficiency.

\paragraph{Necessity.}

We turn to the necessity of the theorem's condition. We again make use of the cochain $s = (s^{(0)}, \ldots, s^{(q)}, 0, \ldots , 0) = \left( \omega^{(0)} - \partial\vartheta^{(0)} , \ldots , \omega^{(q - 1)} - \partial\vartheta^{(q - 1)} , \omega^{(q)}, 0, \ldots , 0 \right)$ output by Construction \ref{construction}.

We explain the idea of our construction below. Provided we write $\varsigma^{(0)} \coloneqq \tilde{s}^{(0)} + \vartheta^{(0)}$, our hypothesis on $\tilde{s}$ immediately establishes that $\left. \omega_A \right|_U - \partial \varsigma^{(0)} = \left( \left. \omega_A \right|_U - \partial\vartheta^{(0)} \right) - \partial \left( \varsigma^{(0)} - \vartheta^{(0)} \right) = s^{(0)} - \partial\tilde{s}^{(0)}$ is holomorphic. The Čech 0-cochain $\varsigma^{(0)}$ thus gives us something close to what we want; unfortunately, we have no reason to hope for that cochain's Čech-closedness. Instead, we express $\varsigma^{(0)}$---or more precisely, a perturbation of that quantity of the shape $\varsigma^{(0)} + \xi^{(0)}$, for $\xi^{(0)}$ holomorphic---as the sum of a $\delta$-closed cochain and a $\partial$-closed cochain. That desideratum is enough to imply our goal, since, it being accomplished, we may use the first of those two summands as $\beta$. On the other hand, we can achieve it only if the Čech differential of $\varsigma^{(0)} + \xi^{(0)}$ is itself the holomorphic differential of a quantity which has already been similarly split. This requirement leads to a recursive procedure of correction that ripples through the Čech double complex (actually two, since we first need to create the adjustment $\xi^{(0)}$ that makes this procedure possible).




Below, we use the notational device whereby, for each $\ell \in \{n, \ldots , 0\}$, we define $\mathcal{K}_\ell$ to be $\Omega^{n - \ell}_{\mathbb{P}^n}(\log X)$ if $\ell > q$ and $\Omega^{n - \ell}_{\mathbb{P}^n}((q + 1 - \ell)X)$ otherwise. We note that these sheaves are locally free of finite rank, and in particular coherent.
\begin{construction}[Construction of de Rham Primitive] \label{necessity}
\textit{Input.}  
A total cochain $\tilde{s} = (\tilde{s}^{(0)}, \ldots , \tilde{s}^{(n - 1)})$ in $\pi_{n - 1}(C_U)$ such that $\left. s \right|_U - d\tilde{s}$ is everywhere-holomorphic. \\
\textit{Output.} An element $\beta \in \Gamma\left( U, \Omega^{n - 1}_{\mathbb{P}^n}(qX) \right)$ such that $\left. \omega_A \right|_U - \partial \beta \in \Gamma\left( U, \Omega^n_{\mathbb{P}^n} \right)$.
\begin{itemize}
\item Define $\left( \varsigma^{(0)}, \ldots , \varsigma^{(n - 1)} \right) \coloneqq \left( \tilde{s}^{(0)}, \ldots , \tilde{s}^{(n - 1)} \right) + \left( \vartheta^{(0)}, (-1)^1 \cdot \vartheta^{(1)}, \ldots , (-1)^{q - 1} \cdot \vartheta^{(q - 1)}, 0, \ldots , 0 \right)$.
\item Initialize $\psi^{(n)} \coloneqq 0$ in $\check{C}^n \left( \mathscr{U} \cap U, \mathcal{O}_{\mathbb{P}^n} \right)$.
\item \textbf{for} $\ell \in \{n - 1, \ldots , 0\}$ \textbf{do}:
\begin{itemize}
\item Since $\check{H}^{\ell + 1}\left( \mathscr{U} \cap U, \Omega^{n - \ell - 1}_{\mathbb{P}^n} \right) = 0$ by the affineness of $U$, and since by induction on $\ell$ (see below) $\psi^{(\ell + 1)} - \delta \varsigma^{(\ell)} \in \check{C}^{\ell + 1}\left( \mathscr{U} \cap U, \Omega^{n - \ell - 1}_{\mathbb{P}^n} \right)$ is both holomorphic and Čech-closed, we can lift that latter element to an element, say $\xi^{(\ell)} \in \check{C}^\ell\left( \mathscr{U} \cap U, \Omega^{n - \ell - 1}_{\mathbb{P}^n} \right)$, under the map
\begin{equation*}\check{C}^\ell\left( \mathscr{U} \cap U, \Omega^{n - \ell - 1}_{\mathbb{P}^n} \right) \xlongrightarrow{\delta} \check{C}^{\ell + 1}\left( \mathscr{U} \cap U, \Omega^{n - \ell - 1}_{\mathbb{P}^n} \right).\end{equation*}
Define $\psi^{(\ell)} \coloneqq -(-1)^\ell \cdot \partial\left( \varsigma^{(\ell)} + \xi^{(\ell)} \right)$ in $\check{C}^\ell\left( \mathscr{U} \cap U, \mathcal{K}_\ell \right)$.
\end{itemize}
\item Initialize $\tau^{(n - 1)} \coloneqq 0$ in $\check{C}^{n - 1} \left( \mathscr{U} \cap U, \mathcal{O}_{\mathbb{P}^n} \right)$.
\item \textbf{for} $\ell \in \{n - 1, \ldots , 1\}$ \textbf{do}:
\begin{itemize}
\item Since $\check{H}^\ell\left( \mathscr{U} \cap U, \mathcal{K}_{\ell + 1} \right) = 0$ by the affineness of $U$, and since, by induction on $\ell$ (see below) $\varsigma^{(\ell)} + \xi^{(\ell)} - \tau^{(\ell)}$ takes values in $\check{C}^\ell\left( \mathscr{U} \cap U, \mathcal{K}_{\ell + 1} \right)$ and is Čech-closed, we can lift that element to an element, say $\rho^{(\ell - 1)} \in \check{C}^{\ell - 1}\left( \mathscr{U} \cap U, \mathcal{K}_{\ell + 1} \right)$, under the map
\begin{equation*}\check{C}^{\ell - 1}\left( \mathscr{U} \cap U, \mathcal{K}_{\ell + 1} \right) \xlongrightarrow{\delta} \check{C}^\ell\left( \mathscr{U} \cap U, \mathcal{K}_{\ell + 1} \right).\end{equation*}
Define $\tau^{(\ell - 1)} \coloneqq (-1)^{\ell - 1} \cdot \partial \rho^{(\ell - 1)}$ in $\check{C}^{\ell - 1} \left( \mathscr{U} \cap U, \mathcal{K}_\ell \right)$.
\end{itemize}
\item Output $\beta \coloneqq \varsigma^{(0)} + \xi^{(0)} - \tau^{(0)}$.
\end{itemize}
\end{construction}
\begin{lemma} \label{necessity_correctness}
The Čech 0-cochain $\beta$ output by Construction \ref{necessity} lives in $\check{C}^0\left( \mathscr{U} \cap U, \Omega^{n - 1}_{\mathbb{P}^n}(qX) \right)$. Further, it's Čech-closed, and so can be viewed as a global section in $\Gamma\left( U, \Omega^{n - 1}_{\mathbb{P}^n}(qX) \right)$. Finally, $\left. \omega_A \right|_U - \partial\beta \in \Gamma\left( U, \Omega^n_{\mathbb{P}^n} \right)$. 
\end{lemma}
\begin{proof}
We begin with Construction \ref{necessity}'s first loop. 
We argue that, for each $\ell \in \{n, \ldots , 0\}$, the following invariants hold throughout that loop.
\begin{enumerate}
\item\label{exact} $\psi^{(\ell)}$ is $\partial$-closed. 
\item\label{value} If $\ell > 0$, then $\psi^{(\ell)} - \delta \varsigma^{(\ell - 1)}$ is holomorphic and $\delta$-closed.
\end{enumerate}
In the base case $\ell = n$, \ref{exact} is trivially true, since $\psi^{(n)} = 0$ by construction. To prove \ref{value}, we note that, since $q < n$, $s^{(n)} = 0$ and $\tilde{s}^{(n - 1)} = \varsigma^{(n - 1)}$. We conclude that $s^{(n)} - \delta \tilde{s}^{(n - 1)} = \psi^{(n)} - \delta \varsigma^{(n - 1)}$; our hypothesis on $\tilde{s}$ guarantees that this item is holomorphic. Its Čech-closedness is self-evident.

To prove the inductive case, we fix an $\ell \in \{n - 1, \ldots , 0\}$. By the inductive case of \ref{value}, $\psi^{(\ell + 1)} - \delta \varsigma^{(\ell)}$ is holomorphic and Čech-closed, so that the construction of $\xi^{(\ell)}$ above makes sense. The $\partial$-closedness of $\psi^{(\ell)}$ is immediate. 
If $\ell > 0$, then, up to the notational convention whereby we write $\omega^{(q + 1)} \coloneqq \cdots \coloneqq \omega^{(n - 1)} \coloneqq 0$ and $\vartheta^{(q)} \coloneqq \cdots \coloneqq \vartheta^{(n - 1)} \coloneqq 0$, we have:
\begin{align*}
\psi^{(\ell)} - \delta\varsigma^{(\ell - 1)} &= -(-1)^\ell \cdot \partial \left( \varsigma^{(\ell)} + \xi^{(\ell)}\right) - \delta\varsigma^{(\ell - 1)} \\
&= \left( \omega^{(\ell)} - \partial\vartheta^{(\ell)} \right) - \delta \left( \varsigma^{(\ell - 1)} - (-1)^{\ell - 1} \cdot \vartheta^{(\ell - 1)} \right) - (-1)^\ell \cdot \partial \left( \varsigma^{(\ell)} - (-1)^\ell \cdot \vartheta^{(\ell)} + \xi^{(\ell)} \right) \\
&= \underbrace{\left( s^{(\ell)} - \delta\tilde{s}^{(\ell - 1)} - (-1)^\ell \cdot \partial \tilde{s}^{(\ell)} \right)}_{\text{holomorphic}} - (-1)^\ell \cdot \underbrace{\left( \partial\xi^{(\ell)} \right)}_{\mathclap{\text{holomorphic}}},
\end{align*}
a quantity whose summands are holomorphic (by our hypothesis on $\tilde{s}$ and by construction, respectively). This proves the first part of \ref{value}. As for that quantity's Čech-closedness, we note that $\delta \left( \psi^{(\ell)} - \delta \varsigma^{(\ell - 1)}\right) = \delta\psi^{(\ell)} = -(-1)^\ell \cdot \delta \partial\left( \varsigma^{(\ell)} + \xi^{(\ell)} \right) = -(-1)^\ell \cdot \partial\left( \delta \varsigma^{(\ell)} + \psi^{(\ell + 1)} - \delta \varsigma^{(\ell)} \right) = -(-1)^\ell \cdot \partial\psi^{(\ell + 1)} = 0$; in the last equality, we use the inductive case of \ref{exact}.

We turn to the second of Construction \ref{necessity}'s two loops. We claim that, for each $\ell \in \{n, \ldots , 1\}$, the following inductive invariants hold with respect to that loop.
\begin{enumerate}
\item\label{tau_closed} $\tau^{(\ell - 1)} \in \check{C}^{\ell - 1} \left( \mathscr{U} \cap U, \mathcal{K}_\ell \right)$ is $\partial$-closed.
\item\label{delta_closed} $\varsigma^{(\ell - 1)} + \xi^{(\ell - 1)} - \tau^{(\ell - 1)} \in \check{C}^{\ell - 1} \left( \mathscr{U} \cap U, \mathcal{K}_\ell \right)$ is $\delta$-closed.
\end{enumerate}
In particular, the two conditions above immediately yield the sum decomposition
\begin{equation*}\varsigma^{(\ell - 1)} + \xi^{(\ell - 1)} = \underbrace{\left( \varsigma^{(\ell - 1)} + \xi^{(\ell - 1)} - \tau^{(\ell - 1)} \right)}_{\text{$\delta$-closed}} + \underbrace{\left( \tau^{(\ell - 1)} \right)}_{\text{$\partial$-closed}},\end{equation*}
valid for each $\ell \in \{n, \ldots , 1\}$.

In the base case $\ell = n$, $\tau^{(n - 1)} = 0$ by definition; we conclude that $\delta \left( \varsigma^{(n - 1)} + \xi^{(n - 1)} - \tau^{(n - 1)}\right) = \delta \varsigma^{(n - 1)} + \psi^{(n)} - \delta \varsigma^{(n - 1)} = \psi^{(n)} = 0$, where in the first equality we use the defining property of $\xi^{(n - 1)}$. 

We now let $\ell \in \{n - 1, \ldots , 1\}$. 
By the inductive case of \ref{delta_closed}, $\varsigma^{(\ell)} + \xi^{(\ell)} - \tau^{(\ell)}$ is $\mathcal{K}_{\ell + 1}$-valued, so that $\rho^{(\ell - 1)}$ too is. This shows that $\tau^{(\ell - 1)}$, which in any case is of course $\partial$-closed, takes values in $\mathcal{K}_\ell$. This proves \ref{tau_closed}. 
We note that $\varsigma^{(\ell - 1)}$ takes values in $\mathcal{K}_\ell$, essentially by Construction \ref{construction}; moreover, $\xi^{(\ell - 1)}$ is holomorphic. This proves that $\varsigma^{(\ell - 1)} + \xi^{(\ell - 1)} - \tau^{(\ell - 1)}$ takes values in $\mathcal{K}_\ell$. To show that cochain's closedness, we note that $\delta \tau^{(\ell - 1)} = (-1)^{\ell - 1} \cdot \partial \delta \rho^{(\ell - 1)} = (-1)^{\ell - 1} \cdot \partial \left( \varsigma^{(\ell)} + \xi^{(\ell)} - \tau^{(\ell)} \right) = -(-1)^{\ell} \cdot \partial \left( \varsigma^{(\ell)} + \xi^{(\ell)} \right) = \psi^{(\ell)} = \delta \xi^{(\ell - 1)} + \delta \varsigma^{(\ell - 1)}$, where the last equality is the defining property of $\xi^{(\ell - 1)}$. In the middle equality, we use the inductive case of \ref{tau_closed}.

We're now ready to prove the lemma. Carrying through the induction, we conclude that $\varsigma^{(0)} + \xi^{(0)} - \tau^{(0)} \in \check{C}^0 \left( \mathscr{U} \cap U, \Omega^{n - 1}_{\mathbb{P}^n}(qX) \right)$ is Čech-closed, and so represents a well-defined section $\beta \in \Gamma\left( U, \Omega^{n - 1}_{\mathbb{P}^n}(qX) \right)$. Finally,
\begin{align*}
\left. \omega_A \right|_U - \partial\beta &= \omega^{(0)} - \partial\left( \varsigma^{(0)} + \xi^{(0)} - \tau^{(0)} \right) \tag{by definition of $\beta$.} \\
&= \left( \omega^{(0)} - \partial \vartheta^{(0)} \right) - \partial \left( \varsigma^{(0)} - \vartheta^{(0)} + \xi^{(0)} - \tau^{(0)} \right) \tag{add and subtract $\partial \vartheta^{(0)}$.} \\
&= s^{(0)} - \partial\left( \tilde{s}^{(0)} + \xi^{(0)} - \tau^{(0)} \right) \tag{by the definitions of $s^{(0)}$ and $\varsigma^{(0)}$.} \\
&= \underbrace{\left( s^{(0)} - \partial \tilde{s}^{(0)} \right)}_{\text{holomorphic}} - \underbrace{\left( \partial\xi^{(0)} \right)}_{\mathclap{\text{holomorphic}}}. \tag{using the $\partial$-closedness of $\tau^{(0)}$.}
\end{align*}
The left-hand summand above is holomorphic by our hypothesis on $\tilde{s}$; the right-hand summand is holomorphic by construction. This proves that $\left. \omega_A \right|_U - \partial\beta$ is holomorphic, and so completes the proof of the lemma.
\end{proof}
The proof of Lemma \ref{necessity_correctness} completes the proof of the theorem.
\end{proof}

In the figure below, we depict the data layout used in this section (especially in Lemmas \ref{correctness_primitive} and \ref{necessity_correctness}).

\insertfullpagefiguredummy

\subsection{Algebraic Differential Forms} \label{algebraic}

In this subsection, we discuss algebraic (i.e., rational) differential forms. First, we rederive a formula stated by Griffiths (see Lemma \ref{griffiths} below). After that, using the Grothendieck comparison theorem, we prove that the differential form $\beta$ of Theorem \ref{main} may be freely assumed algebraic (see Theorem \ref{grothendieck}). 
We follow the notation of Lewis \cite[Lec.~9]{Lewis:1999aa}. We work in $\mathbb{C}^{n + 1}$, with volume form $\operatorname{vol} \coloneqq dZ_0 \wedge \cdots \wedge dZ_n$. We write $\nu \coloneqq \sum_{i = 0}^n Z_i \cdot \frac{\partial}{\partial Z_i}$ for the \textit{Euler} vector field on $\mathbb{C}^{n + 1}$ and $\Omega \coloneqq \left< \nu, \operatorname{vol} \right> = \sum_{i = 0}^n (-1)^i Z_i \cdot dZ_0 \wedge \cdots \wedge \widehat{dZ_i} \wedge \cdots \wedge dZ_n$. 

By \cite[9.17.~Cor.]{Lewis:1999aa}, each rational $n$-form on $\mathbb{P}^n$ admits an expression of the form $\frac{A \cdot \Omega}{B}$, where $A$ and $B$ are homogeneous and $\deg B = n + 1 + \deg A$. Moreover, for $V \coloneqq \left( \frac{A_0}{B}, \ldots , \frac{A_n}{B} \right)$ an arbitrary tuple of homogeneous rational functions in $n + 1$ variables, each of degree $-n$, 
the expression
\begin{equation}\label{eta}\eta \coloneqq \frac{1}{B} \cdot \sum_{i < j} \left( Z_i \cdot A_j - A_i \cdot Z_j \right) \cdot \Omega_{ij},\end{equation}
where $\Omega_{ij} \coloneqq \left< \frac{\partial}{\partial Z_i}, \left< \frac{\partial}{\partial Z_j}, \operatorname{vol} \right> \right>$, represents a well-defined $n - 1$-form on $\mathbb{P}^n$, and conversely every rational $n - 1$-form on $\mathbb{P}^n$ arises in this way.

We now fix a smooth hypersurface $X \subset \mathbb{P}^n$ of degree $d$ and a pair $p + q = n - 1$. If the pole order of $\eta$ above along $X$ is at most $q$, say, then, up to adding to $V$ some rational multiple of the Euler field (which has no effect on $\eta$), we can further write $B = F^q \cdot D$, for some denominator $D \notin (F)$, as well as $V = \left( \frac{v_0}{F^q}, \ldots , \frac{v_n}{F^q} \right)$, where, for each $i \in \{0, \ldots , n\}$, $v_i \coloneqq \frac{A_i}{D}$ is \textit{itself} a rational function, of total degree $d \cdot q - n$. We further define $v \coloneqq \left( v_0, \ldots , v_n \right)$, so that $V = \frac{v}{F^q}$. 
We identify $v = (v_0, \ldots , v_n)$ with $v \coloneqq \sum_{i = 0}^n v_i \cdot \frac{\partial}{\partial Z_i}$, and also write $\operatorname{div}(v) \coloneqq \sum_{i = 0}^n \frac{\partial v_i}{\partial Z_i}$ for $v$'s \textit{divergence}.


We note first that, by manipulating the expression \eqref{eta}, we may reëxpress $F^q \cdot \eta = \sum_{i < j} \left( Z_i \cdot v_j - v_i \cdot Z_j \right) \cdot \left< \frac{\partial}{\partial Z_i}, \left< \frac{\partial}{\partial Z_j}, \operatorname{vol} \right> \right> = \left< \nu, \sum_{i = 0}^n v_i \cdot \left< \frac{\partial}{\partial Z_i}, \operatorname{vol} \right> \right> = -\sum_{i = 0}^n v_i \cdot \left< \frac{\partial}{\partial Z_i}, \Omega \right> = -\left< v, \Omega \right>$. The first equality is essentially \cite[9.19.~Ex]{Lewis:1999aa}; to obtain the second, we interchange the first expression's contractions by $\nu$ and $\frac{\partial}{\partial Z_i}$ (a similar expression appears within the proof of \cite[Lem.~6.11]{Voisin:2003aa}). Moreover, we have:
\begin{lemma}[{Griffiths \cite[(4.5)]{Griffiths:1969aa}}] \label{griffiths}
$\partial \eta = \frac{\Omega}{F^{q + 1}} \cdot \left( q \cdot dF(v) - F \cdot \operatorname{div}(v) \right)$.
\end{lemma}
\begin{proof}
Indeed, we calculate:
\begin{align*}
\partial \eta &= \partial \left( \frac{-\left< v, \Omega \right>}{F^q} \right) \tag{just argued.} \\
&= \frac{q \cdot dF \wedge \left< v, \Omega \right>}{F^{q + 1}} - \frac{ \partial \left< v, \Omega \right>}{F^q} \tag{by Leibniz's rule.} \\
&= \frac{q \cdot dF(v) \cdot \Omega}{F^{q + 1}} - \frac{q \cdot d \cdot \left< v, \operatorname{vol} \right>}{F^q} - \frac{\partial \left< v, \Omega \right>}{F^q} \tag{by Cartan's formula; see below.} \\
&= \frac{q \cdot dF(v) \cdot \Omega}{F^{q + 1}} - \frac{q \cdot d \cdot \left< v, \operatorname{vol} \right>}{F^q} - \frac{\operatorname{div}(v) \cdot \Omega}{F^q} + \frac{q \cdot d \cdot \left< v, \operatorname{vol} \right>}{F^q} \tag{by Cartan's formula again; see below.} \\
&= \frac{\Omega}{F^{q + 1}} \cdot \left( q \cdot dF(v) - F \cdot \operatorname{div}(v) \right). \tag{cancel terms.}
\end{align*}
To achieve the third equality, we use Cartan's formula twice. To wit, $dF \wedge \left< v, \Omega \right> = dF(v) \cdot \Omega - \left< v, dF \wedge \Omega \right> = dF(v) \cdot \Omega - \left< v, dF(\nu) \cdot \operatorname{vol} - \left< \nu, dF \wedge \operatorname{vol} \right> \right> = dF(v) \cdot \Omega - d \cdot F \cdot \left< v, \operatorname{vol} \right>$; in the last step, we use Euler's formula to the effect that $dF(\nu) = d \cdot F$. To achieve the second-to-last equality, we use two further variants of Cartan's formula, now pertaining to the Lie derivative $\mathscr{L}_v$ along $v$. Indeed,
\begin{align*}
\partial \left< v, \Omega \right> &= \mathscr{L}_v \Omega - \left< v, \partial \Omega \right> \tag{using Cartan's formula $\mathscr{L}_v \Omega = \partial \left< v, \Omega \right> + \left< v, \partial \Omega \right>$.} \\
&= \mathscr{L}_v \Omega - \left< v, (n + 1) \cdot \operatorname{vol} \right> \tag{using $\operatorname{div}(\nu) \cdot \operatorname{vol} = \mathscr{L}_\nu \operatorname{vol} = \partial \left< \nu, \operatorname{vol} \right> + \left< \nu, \partial \operatorname{vol} \right> = \partial \left< \nu, \operatorname{vol} \right>$.} \\
&= \left< \nu, \mathscr{L}_v \operatorname{vol} \right> + \left< [v, \nu], \operatorname{vol} \right> - \left< v, (n + 1) \cdot \operatorname{vol} \right> \tag{using $\mathscr{L}_v\left< \nu, \operatorname{vol} \right> = \left< \nu, \mathscr{L}_v \operatorname{vol} \right> + \left< [v, \nu], \operatorname{vol} \right>$.} \\
&= \left< \nu, \mathscr{L}_v \operatorname{vol} \right> + \left(n + 1 - q \cdot d \right) \cdot \left< v, \operatorname{vol} \right> - (n + 1) \cdot \left< v, \operatorname{vol} \right>  \tag{using $[v, \nu] = (n + 1 - d \cdot q) \cdot v$.} \\
&= \left< \nu, \operatorname{div}(v) \cdot \operatorname{vol} \right> - q \cdot d \cdot \left< v, \operatorname{vol} \right> \tag{using $\mathscr{L}_v \operatorname{vol} = \operatorname{div}(v) \cdot \operatorname{vol}$.} \\
&= \operatorname{div}(v) \cdot \Omega - q \cdot d \cdot \left< v, \operatorname{vol} \right>;
\end{align*}
this calculation completes the proof of the lemma.
\end{proof}

For $\omega_A = \frac{A \cdot \Omega}{F^{q + 1}}$ and $\eta$ as above, Lemma \ref{griffiths} implies that $\omega_A - \partial \eta = \frac{\Omega}{F^{q + 1}} \cdot \left( A - q \cdot dF(v) + F \cdot \operatorname{div}(v) \right)$.
In particular, $\omega_A - \partial \eta$ is holomorphic if and only if, in the local ring $\mathbb{C}[Z_0, \ldots , Z_n]_{(F)}$,
\begin{equation}\label{thing}F^{q + 1} \mid A - q \cdot dF(v) + F \cdot \operatorname{div}(v).\end{equation}

Below, we let $F$ define a smooth hypersurface $X \subset \mathbb{P}^n$ of degree $d$; we
assume $q \in \{1, \ldots , n - 1\}$ (the case $q = 0$ is trivial). We
further let $A$ be homogeneous of degree $d \cdot (q + 1) - (n + 1)$, set $\omega_A \coloneqq \frac{A \cdot \Omega}{F^{q + 1}}$ as above,
and let $U \subset \mathbb{P}^n$ be a principal open for which $X \cap U \neq \varnothing$. 
\begin{theorem} \label{grothendieck}
Suppose that there exists a holomorphic form $\beta \in \Gamma\left( U, \Omega^{n - 1}_{\mathbb{P}^n}(qX) \right)$ such that $\left. \omega_A \right|_U - \partial \beta \in \Gamma\left( U,  \Omega^n_{\mathbb{P}^n} \right)$. Then there exists an algebraic form $\eta \in \Gamma\left( U, \Omega^{n - 1}_{\mathbb{P}^n}(qX) \right)$ such that, again, $\left. \omega_A \right|_U - \partial \eta \in \Gamma\left( U,  \Omega^n_{\mathbb{P}^n} \right)$.
\end{theorem}
\begin{proof}
We write $D$ for the form for which $U = \mathbb{P}^n \setminus V(D)$. Since $X$ is smooth, the forms $\frac{\partial F}{\partial Z_0}, \ldots , \frac{\partial F}{\partial Z_n}$ have no common zero in $\mathbb{P}^n$; the Nullstellensatz implies that $D^N \in J(F)$ holds for some large enough $N$.

\begin{construction}[Construction of Algebraic Primitive] \label{grothendieck_construction}
\textit{Input.} A holomorphic differential form $\beta \in \Gamma\left( U, \Omega^{n - 1}_{\mathbb{P}^n}(qX) \right)$ such that $\left. \omega_A \right|_U - \partial\beta \in \Gamma\left( U,  \Omega^n_{\mathbb{P}^n} \right)$. \\
\textit{Output.} An algebraic differential form $\eta \in \Gamma\left( U, \Omega^{n - 1}_{\mathbb{P}^n}(qX) \right)$ such that $\left. \omega_A \right|_U - \partial\eta \in \Gamma\left( U,  \Omega^n_{\mathbb{P}^n} \right)$.
\begin{itemize}
\item Initialize $\alpha_{q + 1} \coloneqq 0$ in $\Gamma\left( U, \Omega^{n - 1}_{\mathbb{P}^n}(qX) \right)$.
\item \textbf{for} $h \in \{q, \ldots , 1\}$ \textbf{do}:
\begin{itemize}
\item By induction on $h$, $\left. \omega_A \right|_U - \partial \alpha_{h + 1}$ is an algebraic element of $\Gamma \left(U, \Omega^n_{\mathbb{P}^n}((h + 1)X) \right)$; it thus satisfies $\left. \omega_A \right|_U - \partial \alpha_{h + 1} = \frac{a_{h + 1} \cdot \Omega}{F^{h + 1}}$, for some $a_{h + 1} = \frac{a'}{D^k}$, where $a'$ is homogeneous and $k \geq 0$.
\item Up to possibly increasing $k$, we may assume that we have an expression $a' = \sum_{j = 0}^n b_j \cdot \frac{\partial F}{\partial Z_j}$; we define $v_h \coloneqq \frac{1}{h \cdot D^k} \cdot \left( b_0, \ldots , b_n \right)$ and set $\gamma_h \coloneqq \frac{-\left< v_h, \Omega \right>}{F^h}$. Since $h \cdot dF(v_h) = a_{h + 1}$, Lemma \ref{griffiths} gives us $\partial \gamma_h = \partial \left( \frac{-\left< v_h, \Omega \right>}{F^h} \right) = \frac{\Omega}{F^{h + 1}} \cdot \left( h \cdot dF(v_h) - F \cdot \operatorname{div}(v_h)\right) = \frac{a_{h + 1} \cdot \Omega}{F^{h + 1}} - \frac{\operatorname{div}(v_h) \cdot \Omega}{F^h}$,
so that the pole order of $\left. \omega_A \right|_U - \partial \alpha_{h + 1} - \partial \gamma_h$ along $X$ is at most $h$. Update $\alpha_h \coloneqq \alpha_{h + 1} + \gamma_h$.
\end{itemize}
\item Inducting, we see that $\alpha_1 \in \Gamma\left( U, \Omega^{n - 1}_{\mathbb{P}^n}(qX) \right)$ is algebraic and satisfies $\left. \omega_A \right|_U - \partial \alpha_1 \in \Gamma\left( U, \Omega^n_{\mathbb{P}^n}(X) \right)$. In particular, $\partial \beta - \partial \alpha_1 = \left(\omega_A|_U - \partial \alpha_1 \right) - \left( \omega_A|_U - \partial \beta\right)$ has simple (equivalently, logarithmic) poles.
\item Initialize $\mu_q \coloneqq 0$ in $\Gamma\left( U, \Omega^{n - 2}_{\mathbb{P}^n}((q - 1)X) \right)$.
\item \textbf{for} $h \in \{q - 1, \ldots , 1\}$ \textbf{do}:
\begin{itemize}
\item By induction on $h$, $\beta - \alpha_1 - \partial \mu_{h + 1} \in \Gamma\left( U, \Omega^{n - 1}_{\mathbb{P}^n}((h + 1)X) \right)$; moreover, its differential $\partial \beta - \partial \alpha_1$ has logarithmic poles (so \textit{a fortiori} order-$h + 1$ poles), by the above. Thus we can lift it under
\begin{equation*}\Omega_{\mathbb{P}^n}^{n - 2}(hX) \xlongrightarrow{\partial} \Omega_{\mathbb{P}^n}^{n - 1}((h + 1)X) \mathbin{/} \Omega_{\mathbb{P}^n}^{n - 1}(hX) \xlongrightarrow{\partial} \Omega_{\mathbb{P}^n}^n((h + 2)X) \mathbin{/} \Omega_{\mathbb{P}^n}^n((h + 1)X)\end{equation*}
to an element $\chi_h \in \Gamma\left( U, \Omega^{n - 2}_{\mathbb{P}^n}(hX) \right)$, say---\textit{not} in general algebraic---such that $\beta - \alpha_1 - \partial \mu_{h + 1} - \partial \chi_h$ has order-$h$ poles at most. Update $\mu_h \coloneqq \mu_{h + 1} + \chi_h$.
\end{itemize}
\item Now $\operatorname{Res}_X \left( \left. \omega_A \right|_U - \partial \alpha_1 \right) = 
\operatorname{Res}_X\left( \partial \beta - \partial \alpha_1 \right) = \partial \operatorname{Res}_X \left( \beta - \alpha_1 - \partial \mu_1 \right)$;
to obtain the first equality, we use our hypothesis on $\beta$, which implies that $\operatorname{Res}_X(\omega_A|_U - \partial \beta) = 0$.
We conclude that $\operatorname{Res}_X \left( \left. \omega_A \right|_U - \partial \alpha_1 \right) \in \Gamma\left( X \cap U, \Omega^{n - 1}_X \right)$
is algebraic and exact. 
By the Grothendieck comparison theorem (see André, Baldassarri, and Cailotto \cite[Cor.~31.1.2]{Andre:2020aa}),
there exists an algebraic primitive $\kappa \in \Gamma\left(X \cap U, \Omega^{n - 2}_X \right)$ such that $\partial \kappa = \operatorname{Res}_X \left( \left. \omega_A \right|_U - \partial \alpha_1 \right)$. By the surjectivity of the residue,
we may find some algebraic $\gamma_0 \in \Gamma\left( U, \Omega^{n - 1}_{\mathbb{P}^n}(\log X) \right)$ such that $\operatorname{Res}_X(\gamma_0) = \kappa$. Setting $\alpha_0 \coloneqq \alpha_1 + \gamma_0$, we obtain $\operatorname{Res}_X \left( \left. \omega_A \right|_U - \partial \alpha_0 \right) = 0$,
so that $\left. \omega_A \right|_U - \partial \alpha_0$ is holomorphic, as desired.
\item Output $\eta \coloneqq \alpha_0$.
\end{itemize}
\end{construction}
The correctness of Construction \ref{grothendieck_construction} is explained inline; that construction establishes the theorem.
\end{proof}

\begin{example} \label{long_example}
For illustration, we discuss the specialization of Construction \ref{construction} to the case $n = 5$, $d = 6$, and $q \in \{1, 2\}$, that of interest to us in Section \ref{solve} below. We fix $\mathscr{U} = \left\{ U_j \right\}_{j \in J} \coloneqq \left\{ D\left( \frac{\partial F}{\partial Z_j}\right) \right\}_{j = 0}^5$ once and for all; we note that this cover's components are affine. We begin with the case $q = 1$. In that case, Construction \ref{construction}, on input $A \in R^6$, begins with the Griffiths form $\omega_A \coloneqq \frac{A \cdot \Omega}{F^2}$ in $\Gamma\left( \mathbb{P}^5, \Omega_{\mathbb{P}^5}^5(2X) \right)$, and constructs from it the Čech 0-cycle $\omega^{(0)} \coloneqq \left\{ \left. \omega_A\right|_{U_{i}} \right\}_{i}$. The total cochain $\left( \omega^{(0)}, 0, \ldots , 0 \right)$ of course doesn't take values in the logarithmic complex $\check{C}(\mathscr{U}, \Omega^\bullet_{\mathbb{P}^5}(\log X))$ ($\omega_A$'s poles are of order 2 along $X$). To correct that cochain, Construction \ref{construction} runs its inner loop exactly once. That inner loop's exact sequence specializes in this case---i.e., in the case $\ell = 1$ and $h = 1$---to
\begin{equation*}\Omega^4_{\mathbb{P}^5}(X) \to \Omega^5_{\mathbb{P}^5}(2X) \mathbin{/} \Omega^5_{\mathbb{P}^5}(X) \to 0;\end{equation*}
the inner loop's body prescribes that a 0-cochain $\gamma_1^{(0)} \in \check{C}^0\left( \mathscr{U}, \Omega_{\mathbb{P}^5}^4(X) \right)$ be chosen in such a way that $\omega^{(0)} - \partial\gamma_1^{(0)}$ has but simple poles. (That we can necessarily do this without refining the cover again owes to the affineness of that cover's components and Cartan's Theorem B.)
That loop then sets $\alpha_1^{(0)} \coloneqq \gamma_1^{(0)}$ and $\vartheta^{(0)} \coloneqq \alpha_1^{(0)}$. Finally, the algorithm sets $\omega^{(1)} \coloneqq -\delta \vartheta^{(0)}$, outputs $\left( \omega^{(0)} - \partial \vartheta^{(0)}, \omega^{(1)}, 0, \ldots , 0 \right)$ and terminates. That total cocycle indeed takes values in the logarithmic complex. Indeed, first of all, $\omega^{(0)} - \partial \vartheta^{(0)}$'s poles are simple by the construction of $\vartheta^{(0)}$, and ``simple poles'' and ``logarithmic poles'' along $X$ coincide for 5-forms on $\mathbb{P}^5$. As for $\omega^{(1)}$, as the proof of Lemma \ref{inductive_construction} points out, $\omega^{(1)} = -\delta \vartheta^{(0)}$ of course has simple poles, and in fact so too does $\partial \omega^{(1)} = -\partial \delta \vartheta^{(0)} = \delta \left( \omega^{(0)} - \partial \vartheta^{(0)} \right)$, since, as we just showed, $\omega^{(0)} - \partial \vartheta^{(0)}$ has simple poles (to achieve the last equality, we use the Čech-closedness of $\omega^{(0)}$).

In fact, it's possible to write down the lift $\vartheta^{(0)}$ above---or rather, a suitable such lift---concretely. We fix a component $U_j$, for $j \in \{0, \ldots , 5\}$, of our cover. Using the notation introduced at the beginning of this subsection, we set $D \coloneqq \frac{\partial F}{\partial Z_j}$, so that $B = F \cdot D$, and set $(A_0, \ldots , A_5) \coloneqq \left(0, \ldots , A, \ldots , 0 \right)$, where the element $A$ resides in the $j$\textsuperscript{th} component; we write $v \coloneqq \left( \frac{A_0}{D}, \ldots , \frac{A_5}{D} \right)$ as usual. 
We note that the total degree of $v_j$ is $6 - 5 = d \cdot q - n$, as required.
By the discussion above, the datum $v$ yields a well-defined element---say, $\vartheta^{(0)}_j$---of $\Gamma \left(U_j, \Omega^4_{\mathbb{P}^5}(X) \right)$; concretely, it takes the form $\vartheta^{(0)}_j = -\sum_{i = 0}^5 \frac{v_i}{F} \cdot \left< \frac{\partial}{\partial Z_i}, \Omega \right> = \frac{-A}{F \cdot \frac{\partial F}{\partial Z_j}} \cdot \left< \frac{\partial}{\partial Z_j}, \Omega \right>$. Moreover, since $dF(v) = \sum_{i = 0}^5 v_i \cdot \frac{\partial F}{\partial Z_i} = A$, Lemma \ref{griffiths} immediately implies that
\begin{equation*}\partial \vartheta^{(0)}_j = \frac{\Omega}{F^2} \cdot \left( dF(v) - F \cdot \operatorname{div}(v) \right) = \frac{\Omega}{F^2} \cdot \left( A - F \cdot \operatorname{div}(v) \right) = \frac{A \cdot \Omega}{F^2} - \frac{\operatorname{div}(v) \cdot \Omega}{F}\end{equation*}
on $U_j$.
We see immediately that $\left. \omega_A \right|_{U_j} - \partial \vartheta^{(0)}_j$ has simple poles along $X$, which is what we wanted to show.

In the case $q = 2$, Construction \ref{construction}, on input $A \in R^{12}$, begins again with $\omega_A \coloneqq \frac{A \cdot \Omega}{F^3}$ in $\Gamma\left( \mathbb{P}^5, \Omega^5_{\mathbb{P}^5}(3X) \right)$ and its associated Čech 0-cycle $\omega^{(0)}$. In the first iteration $\ell = 2$ of its outer loop, its inner loop runs twice (i.e., for $h \in \{2, 1\}$), constructing along the way both $\gamma_2^{(0)}$ and $\gamma_1^{(0)}$. Using the same reasoning just used, we can again describe these 0-cochains explicitly. Indeed, to construct $\gamma_2^{(0)}$, we may use, on each open $U_j$, the 4-form in $\Gamma \left( U_j, \Omega^4_{\mathbb{P}^5}(2X) \right)$ associated to the data $D \coloneqq \frac{\partial F}{\partial Z_j}$ (so that $B = F^2 \cdot D$) and $(A_0, \ldots , A_5) \coloneqq \left( 0, \ldots , \frac{1}{2} \cdot A, \ldots , 0 \right)$. This choice causes $\omega^{(0)} - \partial\alpha_2^{(0)}$'s component at $U_j$ to equal $\frac{\Omega}{F^2} \cdot \frac{\partial}{\partial Z_j} \frac{A}{2 \cdot \frac{\partial F}{\partial Z_j}}$. To construct $\gamma_1^{(0)}$, we can use on $U_j$ the data $D \coloneqq \frac{\partial F}{\partial Z_j}$ (so that $B = F \cdot D$) and $(A_0, \ldots , A_5) \coloneqq \left( 0, \ldots , \frac{\partial}{\partial Z_j} \frac{A}{2 \cdot \frac{\partial F}{\partial Z_j}}, \ldots , 0 \right)$.

As a direct calculation shows, the resulting element $\omega^{(1)} \coloneqq -\delta \vartheta^{(0)}$ equals, on each overlap $U_i \cap U_j$,
\begin{equation}\label{annoying}\frac{A}{2 \cdot F^2} \cdot \left( \frac{\left< \frac{\partial}{\partial Z_j}, \Omega \right>}{\frac{\partial F}{\partial Z_j}} - \frac{\left< \frac{\partial}{\partial Z_i}, \Omega \right>}{\frac{\partial F}{\partial Z_i}} \right)\end{equation}
plus something whose poles along $X$ are simple. In the second iteration $\ell = 1$ of Construction \ref{construction}'s outer loop, that construction's inner loop runs just once (i.e., for $h = 1$). We can once again explicitly write down a suitable choice for that inner loop's 1-cochain $\gamma^{(1)}_1$. Indeed, over each overlap $U_i \cap U_j$, it's enough to use the 3-form $\frac{-A}{2 \cdot F \cdot \frac{\partial F}{\partial Z_i} \cdot \frac{\partial F}{\partial Z_j}} \cdot \left< \frac{\partial}{\partial Z_i}, \left< \frac{\partial}{\partial Z_j}, \Omega \right> \right>$. A direct, albeit cumbersome, calculation akin to that carried out during the proof of Lemma \ref{griffiths} shows that the differential of this form equals \eqref{annoying}, modulo first-order poles.
\end{example}


\section{Producing Explicit Primitives} \label{solve}

In this section, and for the rest of the paper, we turn our attention to the algebraic partial differential equation \eqref{thing}. We fix $n = 5$ and $d = 6$ for good. We work in coordinates $(Z_0, \ldots , Z_5) = (U_0, U_1, U_2, V_0, V_1, V_2)$ equipped with the \textit{block swap} involution
\begin{equation*}\sigma : (U_0, U_1, U_2, V_0, V_1, V_2) \mapsto (V_0, V_1, V_2, U_0, U_1, U_2).\end{equation*}
As we have explained, Voisin \cite{Voisin:2015aa} has asked about a case of the generalized Hodge conjecture pertaining to smooth sextics $X \subset \mathbb{P}^5$ invariant under the involution
\begin{equation*}\iota : (X_0, X_1, X_2, Y_0, Y_1, Y_2) \mapsto i \cdot (Y_0, Y_1, Y_2, -X_0, -X_1, -X_2).\end{equation*}
Up to the projective linear change of coordinates
\begin{equation}\label{change}p : (X_0, X_1, X_2, Y_0, Y_1, Y_2) \mapsto (X_0, X_1, X_2, i \cdot Y_0, i \cdot Y_1, i \cdot Y_2) \eqqcolon (U_0, U_1, U_2, V_0, V_1, V_2),\end{equation}
the involutions $\iota$ and $\sigma$ identify with each other, so our consideration rather of $\sigma$ is no loss.
In fact, Voisin begins rather with a third involution, the \textit{sign flip} involution
\begin{equation*}\pi : (L_0, L_1, L_2, M_0, M_1, M_2) \to (L_0, L_1, L_2, -M_0, -M_1, -M_2).\end{equation*}
The involutions $\pi$ and $\iota$ too identify, this time under a further change of coordinates $(X_0, X_1, X_2, Y_0, Y_1, Y_2) \to (L_0, L_1, L_2, M_0, M_1, M_2)$.
While $\pi$ is the simplest and most attractive of the three involutions, the Shioda-type sextics $f(X_0, X_1, X_2) - f(Y_0, Y_1, Y_2)$ do not identify---under the change of coordinates just above---with sextics in $(L_0, L_1, L_2, M_0, M_1, M_2)$ of any particularly simple or recognizable shape. On the other hand, under the change of coordinates \eqref{change}, these sextics assume the shape $f(U_0, U_1, U_2) + f(V_0, V_1, V_2)$.
For this reason, we work throughout this section in $(U_0, U_1, U_2, V_0, V_1, V_2)$ coordinates, and with the involution $\sigma$. 

We fix an arbitrary smooth sextic $X \subset \mathbb{P}^5$ defined by a $\sigma$-invariant equation $F(U_0, U_1, U_2, V_0, V_1, V_2)$. As Voisin explains, the induced action $\sigma^* : H^4(X, \mathbb{C}) \to H^4(X, \mathbb{C})$ acts as $-1$ on $H^{4, 0}(X) \subset H^4(X, \mathbb{C})$. Indeed, it follows from Griffiths's theory (see again Lewis \cite[Lec.~9]{Lewis:1999aa}) that the nowhere-vanishing Calabi--Yau $4$-form $\omega \in H^{4, 0}(X)$ arises as $\operatorname{Res}_X \left( \frac{\Omega}{F} \right)$, where, again, $\Omega \coloneqq \left< \nu, \operatorname{vol} \right>$ 
is the contraction of the volume form on $\mathbb{C}^6$ with the Euler vector field. Since $F$ is $\sigma$-invariant by hypothesis and $\operatorname{Res}_X$ is $\sigma$-equivariant, it suffices to ascertain how $\Omega = \left< \nu, \operatorname{vol} \right>$ transforms under $\sigma^*$. Since, under each linear automorphism $\Sigma : \mathbb{C}^6 \to \mathbb{C}^6$, the Euler vector field $\nu = \sum_{i = 0}^5 Z_i \cdot \frac{\partial}{\partial Z_i}$ is preserved, we have $\sigma^* \Omega = \sigma^* \left< \nu, \operatorname{vol} \right> = \left< \nu, \sigma^* \operatorname{vol} \right> = \det\Sigma \cdot \Omega$; moreover, for our particular involution $\sigma$, we further have $\det\Sigma = \det \left(\left[\begin{array}{c|c}
\phantom{I_3} & I_3 \\
\hline
I_3 & \phantom{I_3}
\end{array}\right] \right) = -1$.
This completes our demonstration that $H^4(X, \mathbb{Q})^+ \subset H^4(X, \mathbb{Q})$ is of Hodge coniveau at least 1 (see \cite[7.11.~Def]{Lewis:1999aa} for definitions). More generally, the same reasoning shows that, for each $q \in \{0, \ldots , 4\}$ and each $A \in R^{6 \cdot q}$, the differential form $\omega_A = \frac{A \cdot \Omega}{F^{q + 1}}$ is $\sigma$-invariant if and only if $A$ is $\sigma$-\textit{anti-invariant}.

In any case, the generalized Hodge conjecture predicts that, for each $\gamma \in H^4(X, \mathbb{Q})^+$, there should exist a divisor $Y \subset X$ (depending on $\gamma$ in general) such that $\gamma \in \ker\left( H^4(X, \mathbb{Q}) \to H^4(X \setminus Y, \mathbb{Q}) \right)$ (see e.g.\ \cite[7.8.~Prop-Def.]{Lewis:1999aa}). Equivalently, we may ask instead that, for each $\gamma \in H^4(X, \mathbb{C})^+$, there exist a divisor $Y \subset X$ such that $\gamma \in \ker\left( H^4(X, \mathbb{C}) \to H^4(X \setminus Y, \mathbb{C}) \right)$. Indeed, the equivalence of these two formulations is a standard consequence of the functoriality of the tensor product and the fact that $\sigma$ is defined over $\mathbb{Q}$.
By the Lefschetz theorem for Picard groups (see e.g.\ Lazarsfeld \cite[Ex~3.1.25.]{Lazarsfeld:2004aa}), we may freely consider just divisors $Y \subset X$ that arise as $Y = W \cap X$, for $W \subset \mathbb{P}^5$ a hypersurface. By Theorem \ref{containment} and Hodge symmetry, it's enough to treat solely the case $q = 2$. Finally, the non-primitive invariant class $h^2 \in H^4(X, \mathbb{Q})$ presents no issue.

These remarks, together with the results of our Section \ref{basic} (in particular, that section's Theorem \ref{containment}, Theorem \ref{main}, and Theorem \ref{grothendieck}) imply that the prediction of the generalized Hodge conjecture as regards $H^4(X, \mathbb{Q})^+ \subset H^4(X, \mathbb{Q})$, for $X \subset \mathbb{P}^5$ a smooth, $\sigma$-invariant sextic fourfold, is \textit{equivalent} to the following question. 
\begin{question} \label{question}
Let $F(U_0, U_1, U_2, V_0, V_1, V_2)$ be a $\sigma$-invariant sextic whose vanishing locus is smooth; fix $q = 2$. For each $\sigma$-anti-invariant $A \in R^{6 \cdot q}$, does there exist a vector field $v = (v_0, \ldots , v_5)$, whose components are degree-$6 \cdot q - 5$ homogeneous rational functions with denominators not divisible by $F$, for which \eqref{thing} holds?
\end{question}

The main purpose of this section is to answer Question \ref{question} for $F(U_0, U_1, U_2, V_0, V_1, V_2)$ of the form $f(U_0, U_1, U_2) + f(V_0, V_1, V_2)$, for $f(U_0, U_1, U_2)$ a plane sextic of Waring rank 3 (see Theorem \ref{last}).


\subsection{The Reduced Algorithm} \label{reduced_subsection}

The following algorithm plays a key internal role in our solution. It solves \eqref{thing} for $F(U_0, U_1, U_2, V_0, V_1, V_2) = U_0^6 + U_1^6 + U_2^6 + V_0^6 + V_1^6 + V_2^6$ the Fermat sextic, $q \in \{1, 2\}$, and $A \in R^{6 \cdot q}$ an \textit{arbitrary} homogeneous polynomial (not necessarily $\sigma$-anti-invariant) each of whose monomials contains indeterminates raised to the power 4 at most. For $I \subset \{0, \ldots , 5\}$ a subset and $a = (a_0, \ldots , a_5)$ an exponent tuple, we use the notation $w_I(a) \coloneqq \sum_{i \in I} a_i$; we also write $\nu_I \coloneqq \sum_{i \in I} Z_i \cdot \frac{\partial}{\partial Z_i}$ for the ``partial Euler'' vector field determined by $I$. In this section, we use the notational abbreviation $v(F) \coloneqq dF(v) = \sum_{i = 0}^5 v_i \cdot \frac{\partial F}{\partial Z_i}$.

\begin{construction}[Solution For Jacobian-Reduced Input] \label{reduced}
\textit{Input.} An integer $q \in \{1, 2\}$ and a polynomial $A \in R^{6 \cdot q}$ whose monomials' exponents are all at most 4. \\
\textit{Output.} A solution $v = (v_0, \ldots , v_5)$ solving \eqref{thing} for $A$ and for $F(U_0, U_1, U_2, V_0, V_1, V_2)$ the Fermat sextic.
\begin{itemize}
\item Initialize $v \coloneqq 0$.
\item \textbf{for each} monomial $c_a \cdot Z^a$ in $A$ \textbf{do}:
\begin{itemize}
\item Find a nonempty and proper subset $I \subset \{0, \ldots , 5\}$ such that $w_I(a) + |I| = 6$.
\item Update $v \mathrel{+}= c_a \cdot Z^a \cdot \frac{\nu_I}{q \cdot \nu_I(F)}$.
\end{itemize}
\item Output $v$.
\end{itemize}
\end{construction}

\begin{theorem} \label{reduced_proof}
For $F(U_0, U_1, U_2, V_0, V_1, V_2)$ the Fermat sextic, each $q \in \{1, 2\}$, and each input $A \in R^{6 \cdot q}$ fulfilling the requirement set forth in Construction \ref{reduced}, that construction's output $v = (v_0, \ldots , v_5)$ solves \eqref{thing}.
\end{theorem}
\begin{proof}
We proceed with the aid of a handful of lemmas.
The following lemma shows that the key line of Construction \ref{reduced} will actually succeed, so that that construction makes sense.
\begin{lemma} \label{combinatorial}
For each $q \in \{1, 2\}$ and each term $a = (a_0, \ldots , a_5)$ such that $a_i \in \{0, \ldots ,4 \}$ for each $i \in \{0, \ldots , 5\}$ and $\sum_{i = 0}^5 a_i = 6 \cdot q$, there exists a nonempty proper subset $I \subset \{0, \ldots , 5\}$ such that $w_I(a) + |I| = 6$.
\end{lemma}
\begin{proof}
The lemma's statement can easily be checked by brute enumeration. We record a minimal proof. Up to defining $b_i \coloneqq a_i + 1$ for each $i \in \{0, \ldots , 5\}$, we may equivalently ask instead whether each tuple $b = (b_0, \ldots , b_5)$ of \textit{positive} integers $b_i \in \{1, \ldots , 5\}$ whose sum is $6 \cdot (q + 1)$ admits a subset sum $\sum_{i \in I} b_i = 6$, for some index subset $I \subset \{0, \ldots , 5\}$ (itself necessarily nonempty and proper).

\paragraph{Case $q = 1$.} We proceed by case enumeration. If $(b_0, \ldots , b_5)$ contains a 5, then it necessarily also contains a 1 (or else its sum would be at least $5 + 5 \cdot 2 = 15 > 12$); these two elements yield $I$. Similarly, if $\max(b_0, \ldots , b_5) = 4$, then $b$ necessarily contains at least \textit{two} 1s (or else its sum would be at least $4 + 1 + 4 \cdot 2 = 13 > 12$); the resulting three elements work. In the case $\max b = 3$, if $b$ contains a $2$, then it must also contain a $1$ (or else its sum would be at least $3 + 5 \cdot 2 = 13 > 12$), so that any subset consisting of a $3$, a $2$ and a $1$ works; otherwise, $b$ consists of exactly three $3$s and exactly three $1$s, and any subset consisting of two of the three $3$s suffices. Finally, if $\max b = 2$, then we must have $b = (2, 2, 2, 2, 2, 2)$ and any size-3 subset $I$ suffices.

\paragraph{Case $q = 2$.} We again enumerate cases. If $(b_0, \ldots , b_5)$ contains both a $5$ and a $1$, then we're done. If $\max b = 5$ and $b$ \textit{doesn't} contain a $1$, then the remaining five items are in $\{2, 3, 4, 5\}$, and further add to $13$. Each latter such sequence (in fact, there are only three, up to permutation) necessarily contains either three $2$s or two $3$s; these suffice for $I$. In the case $\max b \leq 4$, if there are three $2$s, a $4$ and a $2$, or two $3$s, then we're done. If none of these conditions hold, then the counts $c_j \coloneqq \left| \set*{i \in \{0, \ldots , 5\} ; b_i = j} \right|$ are such that $c_3 \leq 1$ and $c_2 \leq 2$, as well as that $c_4$ and $c_2$ aren't \textit{both} positive. Finally, $\sum_{i = 1}^4 c_i = 6$ and $\sum_{i = 1}^4 i \cdot c_i = 18$. Subtracting thrice the first of these two equations from the second, we obtain $c_4 = 2 \cdot c_1 + c_2$. If $c_2 > 0$ held, then we'd obtain $0 = c_4 = 2 \cdot c_1 + c_2 > 0$; we thus conclude $c_2 = 0$, whence $c_4 = 2 \cdot c_1$, so that the first sum above yields $3 \cdot c_1 + c_3 = 6$. Using $c_3 \leq 1$, we conclude finally that $c_3 = 0$, whence $c_1 = 2$ and $c_4 = 4$. In other words, $b$ consists of four $4$s and two $1$s; taking any of the $4$s together with the two $1$s suffices for $I$. 

These two cases taken together suffice to establish the lemma.
\end{proof}

The following lemma shows that the denominator $\nu_I(F)$ is allowed.
\begin{lemma} \label{lie}
For each nonempty and proper subset $I \subset \{0, \ldots , 5\}$, $\nu_I(F) \notin (F)$.
\end{lemma}
\begin{proof}
We write $\operatorname{Lin} X \subset \text{PGL}_6(\mathbb{C})$ for the group of projective linear automorphisms of $X$. It is classical that this very group is finite (see e.g.\ Poonen \cite[Thm.~1.3]{Poonen:2005aa}). We conclude that the Lie algebra $\operatorname{Lie}(\operatorname{Lin}\, X) = 0$.

For contradiction, we assume that $I \subset \{0, \ldots , 5\}$, nonempty and proper, is such that $\nu_I(F) \in (F)$. For degree reasons, we necessarily have $\nu_I(F) = c \cdot F$ for some constant $c$. We denote by $(\varepsilon_0, \ldots , \varepsilon_5) \coloneqq \left( \mathbf{1}_I(0), \ldots , \mathbf{1}_I(5) \right)$ the membership indicators of $I$. The diagonal matrix $A_I \coloneqq \text{diag}\left(\varepsilon_0, \ldots , \varepsilon_5\right)$ in $\mathfrak{gl}_6(\mathbb{C})$ yields the one-parameter subgroup $g_t \coloneqq \operatorname{exp}(t \cdot A_I) = \operatorname{diag} \left( e^{\varepsilon_0 \cdot t}, \ldots , e^{\varepsilon_5 \cdot t} \right)$ of $\text{GL}_6(\mathbb{C})$. Moreover, writing $F_t(Z) \coloneqq F(g_t \cdot Z)$, we obtain from the chain rule that
\begin{equation*}\frac{d}{dt} F_t(Z) = \sum_{i = 0}^5 \varepsilon_i \cdot e^{\varepsilon_i \cdot t} \cdot Z_i \cdot \frac{\partial F}{\partial Z_i}(g_t \cdot Z) = \sum_{i \in I} \left( g_t \cdot Z \right)_i \cdot \frac{\partial F}{\partial Z_i}(g_t \cdot Z) = \nu_I(F)(g_t \cdot Z) = c \cdot F_t(Z);\end{equation*}
the last equality is our hypothesis on $I$. In particular, $F(g_t \cdot Z) = e^{c \cdot t} \cdot F(Z)$ holds for each $t$. We see that $g_t$ preserves the cone in $\mathbb{C}^6$ over $X$, and hence that $[g_t] \in \operatorname{Lin}\, X$ for each $t$.

Equivalently, the class of $A_I$ in $\mathfrak{pgl}_6(\mathbb{C})$ is tangent at the identity to a one-parameter subgroup of $\operatorname{Lin}\, X$, and so lies in $\operatorname{Lie}(\operatorname{Lin}\, X)$. On the other hand, this class is nonzero; indeed, since $I$ is nonempty and proper, $A_I$ is not a scalar matrix. This fact contradicts our initial conclusion $\operatorname{Lie}(\operatorname{Lin}\, X) = 0$, established above, and so shows that $\nu_I(F) \notin (F)$.
\end{proof}
A conclusion analogous to Lemma \ref{lie}'s in fact holds for each degree-$d$ hypersurface $X \subset \mathbb{P}^n$ for which $n \geq 2$ and $d \geq 3$.

The next lemma is our key correctness claim.

\begin{lemma} \label{quotient}
Each vector field $v_a \coloneqq Z^a \cdot \frac{\nu_I}{q \cdot \nu_I(F)}$ in Construction \ref{reduced} satisfies $q \cdot v_a(F) = Z^a$ and $\operatorname{div}(v_a) = 0$.
\end{lemma} 
\begin{proof}
To prove the first of the lemma's two assertions, we note that
\begin{equation*}q \cdot dF(v_a) = q \cdot Z^a \cdot \sum_{i \in I} \frac{Z_i}{q \cdot \nu_I(F)} \cdot \frac{\partial F}{\partial Z_i} = Z^a.\end{equation*}
To prove the second, we note first of all the partial Euler relation $\nu_I(Z^a) = \sum_{i \in I} Z_i \cdot \frac{\partial Z^a}{\partial Z_i} = w_I(a) \cdot Z^a$; moreover, for $F(U_0, U_1, U_2, V_0, V_1, V_2)$ the Fermat sextic and each $I \subset \{0, \ldots , 5\}$, we have $\nu_I(\nu_I(F)) = 6 \cdot \nu_I(F)$. Using these observations and the quotient rule, we obtain
\begin{equation*}\operatorname{div}(v_a) = \frac{q \cdot Z^a \cdot \left( \left( w_I(a) + |I| \right) \cdot \nu_I(F) - \nu_I(\nu_I(F)) \right)}{\left(q \cdot \nu_I(F) \right)^2} = \frac{Z^a \cdot \nu_I(F) \cdot \left( w_I(a) + |I| - 6 \right)}{q \cdot \left(\nu_I(F) \right)^2} = 0;\end{equation*}
to achieve the final equality, we use the defining property $w_I(a) + |I| = 6$ of $I$.
\end{proof}
Lemmas \ref{lie} and \ref{quotient} immediately imply that for each $q \in \{1, 2\}$ and each monomial $Z^a$ satisfying the condition put forth in Construction \ref{reduced}, the summand vector field $v_a$ solves \eqref{thing} for $F$ the Fermat sextic on the input $Z^a$. This fact implies the statement of the theorem, by the linearity of \eqref{thing}.
\end{proof}

\begin{remark} \label{fails}
We note that Theorem \ref{reduced_proof}'s Jacobian-reduction assumption---i.e., whereby $a_i \leq 4$ for each $i \in \{0, \ldots , 5\}$---is necessary. 
Indeed, the example $a = (6, 0, 0, 0, 0, 0)$ shows that, already in the case $q = 1$, the conclusion of Lemma \ref{combinatorial} fails if the assumption $a_i \leq 4$ is relaxed. If that assumption weren't necessary, then the subsection \textit{after} this one would be unnecessary. We also note that the technique of this work seems not to work unchanged for higher-dimensional Calabi--Yau hypersurfaces. Indeed, to prove the analogue of this work's result for the Fermat sixfold $X \subset \mathbb{P}^7$ of degree 8, we would need to find, for each $q \in \{1, 2, 3\}$ and each tuple $a = (a_0, \ldots , a_7)$ such that $a_i \in \{0, \ldots, 6\}$ holds for each $i \in \{0, \ldots , 7\}$ and $\sum_{i = 0}^7 a_i = 8 \cdot q$, a nonempty proper subset $I \subset \{0, \ldots , 7\}$ such that $\sum_{i \in I} a_i + |I| = 8$. Already in the case $q = 2$, this becomes impossible; for example, the tuple $a = (2, 2, 2, 2, 2, 2, 2, 2)$ shows as much (no subset of $b = (3, 3, 3, 3, 3, 3, 3, 3)$ sums to $8$). Finally, Lemma \ref{combinatorial} fails if we allow $q = 3$. Indeed, $a = (3, 3, 3, 3, 3, 3)$ is a counterexample (no subset of $b = (4, 4, 4, 4, 4, 4)$ sums to 6). This last fact shows that we really must use Hodge symmetry; our method can't explicitly treat all indices $q \in \{0, \ldots , 4\}$. Separately, we note that we have not assumed $A$ $\sigma$-anti-invariant at all in this subsection; that assumption is used crucially in the sequel (see Lemma \ref{final}).
\end{remark}

\subsection{The Full Algorithm} \label{full}

In this subsection, we fully solve \eqref{thing} for $F$ the Fermat sextic, making use of the algorithm of Subsection \ref{reduced_subsection}. We fix $F(U_0, U_1, U_2, V_0, V_1, V_2) = U_0^6 + U_1^6 + U_2^6 + V_0^6 + V_1^6 + V_2^6$ for the entirety of this subsection.

For preliminaries in multivariate division, we refer to von zur Gathen and Gerhard \cite{Gathen:2013aa}. In particular, in the construction below, we fix an \textit{arbitrary} monomial order on $\mathbb{C}[Z_0, \ldots , Z_5]$ (see \cite[Def.~21.4]{Gathen:2013aa}); we moreover make use of multivariate division with remainder (see \cite[Alg.~21.11]{Gathen:2013aa}). 

We introduce a convenient notational device. For each $q > 0$, we write $T : R^{6 \cdot q} \to R^{6 \cdot (q - 1)}$ for the map that, on input $A \in R^{6 \cdot q}$, first multivariate-divides $\frac{A}{q}$ with respect to $\left( \frac{\partial F}{\partial Z_0}, \ldots , \frac{\partial F}{\partial Z_5} \right)$, so obtaining $\frac{A}{q} \eqqcolon \sum_{i = 0}^5 v'_i \cdot \frac{\partial F}{\partial Z_i} + r$ say, and then outputs $\operatorname{div}(v')$. We write $T^{(q)}$ for the $q$-fold iterate of $T$.

In the following lemma, we show that each $\sigma$-anti-invariant polynomial $A$ iterates under $T$ to 0.
\begin{lemma} \label{final}
For each integer $q \geq 0$ and each $\sigma$-anti-invariant polynomial $A \in R^{6 \cdot q}$, we have $T^{(q)}(A) = 0$.
\end{lemma}
\begin{proof}
Because each $\frac{\partial F}{\partial Z_i} = 6 \cdot Z_i^5$ is a monomial, the action of multivariate division proceeds independently on each monomial of $A$. We thus first study the case in which $A = Z^a = Z_0^{a_0} \cdot \cdots Z_5^{a_5}$ is a monomial. In this case, the composed mapping $T : A \mapsto v' \mapsto \operatorname{div}(v')$ operates as follows. If $a_i \leq 4$ for each $i \in \{0, \ldots , 5\}$, then $T : Z^a \mapsto (0, \ldots , 0) \mapsto 0$; otherwise, writing $i(a) \coloneqq \min\set*{i \in \{0, \ldots , 5 \}; a_i \geq 5}$, we obtain
\begin{equation*}T : Z^a \mapsto \left( 0, \ldots , \frac{Z^{a - 5 \cdot e_{i(a)}}}{6 \cdot q}, \ldots , 0 \right) \mapsto \frac{a_{i(a)} - 5}{6 \cdot q} \cdot Z^{a - 6e_{i(a)}}\end{equation*}
unless $a_{i(a)} = 5$, in which case again $T(Z^a) = 0$. Above, we write $e_{i(a)}$ for the $i(a)$\textsuperscript{th} standard basis vector.
Decomposing, for each $i \in \{0, \ldots , 5\}$, $a_i \eqqcolon 6 \cdot b_i + r_i$, where $r_i \in \{0, \ldots , 5\}$, we see that $Z^a$ survives $T$ for at most $\sum_{i = 0}^5 b_i$ iterations of that map (and no more).

If, for some $i \in \{0, \ldots , 5\}$, $r_i > 0$ holds, then from $6 \cdot q = \sum_{i = 0}^5 a_i = 6 \cdot \sum_{i = 0}^5 b_i + \sum_{i = 0}^5 r_i$, we deduce that $\sum_{i = 0}^5 b_i < q$, so that $Z^a$ does \textit{not} survive $q$ iterations of the map. In other words, each term that survives all $q$ iterations necessarily takes the form $Z^{6 \cdot b} = Z_0^{6 \cdot b_0} \cdot \cdots Z_5^{6 \cdot b_5}$, where now $\sum_{i = 0}^5 b_i = q$ (and in fact, this condition is sufficient). Moreover, for each such term, the final value that term attains is a constant, equal to
\begin{equation*}\lambda(b) \coloneqq \frac{1}{q!} \cdot \prod_{i = 0}^5 \underbrace{\left( \frac{6 \cdot b_i - 5}{6} \cdot \frac{6 \cdot b_i - 11}{6} \cdot \cdots \frac{1}{6} \right)}_{\text{$b_i$ terms in product}} = \frac{1}{q!} \cdot \prod_{i = 0}^5 \prod_{j = 0}^{b_i - 1} \frac{6 \cdot j + 1}{6}.\end{equation*}
This expression is clearly permutation-invariant in the tuple $b = (b_0, \ldots , b_5)$. On the other hand, by hypothesis, $A$'s nonzero monomials can be partitioned into pairs of the form $c \cdot Z^a - c \cdot Z^{a'}$, where the tuples $a$ and $a'$ constitute a $\sigma$-orbit, and so \textit{a fortiori} differ by a permutation. Each pair of that form either dies before the $q$\textsuperscript{th} iteration, or else yields $c \cdot \lambda(b) - c \cdot \lambda(b') = 0$; here, we write $a \eqqcolon 6 \cdot b$ and $a' \eqqcolon 6 \cdot b'$, and use the permutation-invariance of $\lambda$. We conclude that $A$ itself dies; this completes the proof of the lemma.
\end{proof}

We now present the main construction of this subsection.

\begin{construction}[Solution For Fermat Sextic] \label{full_construction}
\textit{Input.} An integer $q \in \{0, 1, 2\}$ and a polynomial $A \in R^{6 \cdot q}$ whose $q$-fold iterate satisfies $T^{(q)}(A) = 0$. \\  
\textit{Output.} A solution $v = (v_0, \ldots , v_5)$ solving \eqref{thing} for $A$ and for $F(U_0, U_1, U_2, V_0, V_1, V_2)$ the Fermat sextic.
\begin{itemize}
\item \textbf{if} $q = 0$ \textbf{then return} $v \coloneqq (0, 0, 0, 0, 0, 0)$.
\item Using multivariate division with respect to $\left( \frac{\partial F}{\partial Z_0}, \ldots , \frac{\partial F}{\partial Z_5} \right)$, decompose $\frac{A}{q} \eqqcolon \sum_{i = 0}^5 v'_i \cdot \frac{\partial F}{\partial Z_i} + r$.
\item Call this construction recursively on the inputs $q - 1$ and $\operatorname{div}(v')$, so obtaining $v''$, say.
\item Call Construction \ref{reduced} on the inputs $q$ and $q \cdot r$, so obtaining $v'''$, say.
\item Output $v \coloneqq v' + F \cdot v'' + v'''$.
\end{itemize}
\end{construction}

\begin{theorem} \label{final_correctness}
Construction \ref{full_construction} is correct; that is, on input an integer $q \in \{0, 1, 2\}$ and a polynomial $A \in R^{6 \cdot q}$ satisfying $T^{(q)}(A) = 0$, that construction's output $v$ satisfies \eqref{thing} for $A$, $q$, and the Fermat sextic $F$.
\end{theorem}
\begin{proof}
We prove the theorem by induction. In the base case $q = 0$, we have $T^{(0)}(A) = A = 0$ by hypothesis on $A$; the construction's return value $v \coloneqq (0, 0, 0, 0, 0, 0)$ is thus tautologically correct in that case.

We now let $q \in \{1, 2\}$. We note that, if $A$ satisfies the hypothesis of the construction, then the recursive input $\operatorname{div}(v')$ also does, since $0 = T^{(q)}(A) = T^{(q - 1)} \left( \operatorname{div}(v') \right)$. We may thus invoke the theorem inductively. In this way, we conclude that the recursively constructed quantity $v''$ satisfies
\begin{equation*}F^q \cdot B' = \operatorname{div}(v') - (q - 1) \cdot dF(v'') + F \cdot \operatorname{div}(v'')\end{equation*}
for some homogeneous rational factor $B'$ of degree $-6$ (in fact, $B'$ will be 0 for us). Moreover, by Theorem \ref{reduced_proof}, we have the guarantees $q \cdot dF(v''') = q \cdot r$ (i.e., $dF(v''') = r$) and $\operatorname{div}(v''') = 0$. Using these facts, together with the expression $v \coloneqq v' + F \cdot v'' + v'''$ and linearity, we obtain
\begin{align*}
A - q \cdot dF(v) + F \cdot \operatorname{div}(v) &= A - q \cdot \bigl( dF(v') + F \cdot dF(v'') + dF(v''') \bigr) + F \cdot \bigl( \operatorname{div}(v') + F \cdot \operatorname{div}(v'') + dF(v'') \bigr) \\
&= F \cdot \bigl( \operatorname{div}(v') - (q - 1) \cdot dF(v'') + F \cdot \operatorname{div}(v'') \bigr) \tag{using $A = q \cdot \left( dF(v') + r \right)$.} \\
&= F \cdot F^q \cdot B' \tag{by the recursive identity above.} \\
&= F^{q + 1} \cdot B',
\end{align*}
a calculation that completes the proof of the theorem.
\end{proof}

Combining Lemma \ref{final} and Theorem \ref{final_correctness}, we deduce that Construction \ref{full_construction} solves \eqref{thing} for $q = 2$, each $\sigma$-anti-invariant $A \in R^{12}$, and $F$ the Fermat sextic.

\begin{remark} \label{symmetrization}
While it would have been nice to prove Lemma \ref{final} by arguing directly that the operation $T : A \mapsto \operatorname{div}(v')$ preserves anti-invariance, it in fact doesn't in general. Thus our recursive calls to Construction \ref{full_construction} indeed supply to that construction arguments $A$ that, in general, are \textit{not} anti-invariant, though they still of course satisfy the requirement put forth by that construction.
At the cost of explicitly anti-symmetrizing $v'$ and $r$ at each level of Construction \ref{full_construction}'s recursion, we could have alternatively bypassed Lemma \ref{final} entirely, and in fact achieved something like that lemma's conclusion for \textit{each} $\sigma$-invariant $F$. Since Construction \ref{reduced} of course needs $F$ to be Fermat, this generalization wouldn't have done much for us, and in any case our approach above lets us write down a simpler algorithm (that anti-symmetrization would have turned out superfluous in the Fermat case, as Lemma \ref{final} shows).
\end{remark}

\begin{remark} \label{suitable}
Lemma \ref{final} in fact has nothing to do with $q$, and works for $q \in \{0, \ldots , 4\}$ arbitrary (or presumably even higher). Theorem \ref{final_correctness}'s proof uses its hypothesis $q \in \{0, 1, 2\}$ only to guarantee that Construction \ref{full_construction}'s calls to Construction \ref{reduced} fulfill the hypothesis of Theorem \ref{reduced_proof} (recall also Remark \ref{fails}). Theorem \ref{final_correctness}'s proof otherwise goes through without issue on arbitrary $q \in \{0, \ldots , 4\}$. In fact, something like Construction \ref{full_construction}---augmented, that is, with the symmetrization step sketched in Remark \ref{symmetrization}---would work for arbitrary $\sigma$-invariant sextic fourfolds $X \subset \mathbb{P}^5$ (or even higher-dimensional varieties, in principle), given a suitable replacement for Construction \ref{reduced}.
\end{remark}

\begin{remark} \label{bls}
Our Construction \ref{full_construction} above is related to Bostan, Lairez and Salvy's \textit{Griffiths--Dwork reduction} \cite[Alg.~1]{Bostan:2013aa}, and can be viewed as an enrichment of that algorithm, albeit specific to our Fermat case (though see Remark \ref{suitable} above). It differs from that algorithm in two important ways. For one, Construction \ref{full_construction}'s calls to Construction \ref{reduced}, and indeed that entire latter construction, are new. Further, we prove (i.e., in Lemma \ref{final}) that our algorithm's base case input vanishes (at least if its initial input $A$ is anti-invariant). Bostan, Lairez and Salvy's algorithm \cite[Alg.~1]{Bostan:2013aa} instead accumulates all of its remainders $r$, as well as its base input, and simply reports that $\omega_A$ is exact on $\mathbb{P}^n \setminus X$ if and only if all of these quantities vanish \cite[Thm.~1]{Bostan:2013aa}. In our setting, we must rather decide whether $\omega_A$ becomes exact (up to holomorphic error) on some nontrivial proper open subset $U \setminus X$ of $\mathbb{P}^n \setminus X$; this question is outside the scope of \cite[Alg.~1]{Bostan:2013aa}. 
\end{remark}

\begin{remark}
The solutions to \eqref{thing} that we've managed to attain---i.e., using Construction \ref{full_construction}---solve that equation ``strictly'', in the sense that they cause its right-hand side to vanish entirely (as opposed to merely modulo $F^{q + 1}$). We don't expect that strict solutions should exist for general $\sigma$-invariant sextics $F$; the fact that we haven't managed to find \textit{non-strict} solutions yet suggests that we're leaving something on the table.
\end{remark}

\subsection{The Main Theorem}

This subsection's main theorem assembles the various results we've accumulated.

We begin with a few preliminaries. Given an invertible linear map $M : \mathbb{C}^6 \to \mathbb{C}^6$, we have the usual \textit{pullback} map $M^*$ on $\mathbb{C}[Z_0, \ldots , Z_5]$ (effected by precomposition). Moreover, for $v = (v_0, \ldots , v_5)$ a vector field with components in $\mathbb{C}[Z_0, \ldots , Z_5]_{(F)}$, we define the \textit{pullback} of $v$ as
\begin{equation*}M^*v \coloneqq \begin{bmatrix} \vphantom{\begin{bmatrix} M^*v_0 \\ \vdots \\ M^*v_5 \end{bmatrix}} \mathmakebox[4em][c]{M} \end{bmatrix}^{-1} \cdot \begin{bmatrix} M^* v_0 \\ \vdots \\ M^* v_5\end{bmatrix}.\end{equation*}
This coordinate-change rule coincides with that which transports $v$, interpreted as a \textit{derivation} mapping $\mathbb{C}[Z_0, \ldots , Z_5]_{(F)} \to \mathbb{C}[Z_0, \ldots , Z_5]_{(F)}$, along the isomorphism $M^* : \mathbb{C}[Z_0, \ldots , Z_5]_{(F)} \to \mathbb{C}[Z_0, \ldots , Z_5]_{(M^* F)}$.

This rule plays nicely with the various operations we've been using.
\begin{lemma} \label{pullback}
For each $M$, $F$, and $v$, we have $M^* \left( dF(v) \right) = d\left(M^*F\right) \left( M^*v \right)$ and $M^* \left( \operatorname{div}(v) \right) = \operatorname{div}\left( M^* v \right)$.
\end{lemma}
\begin{proof}
The first assertion is a basic consequence of the chain rule. Indeed,
\begin{equation*}
d\left(M^*F\right) \left( M^*v \right) = \begin{bmatrix}\horzbar\, \frac{\partial}{\partial Z_i} \left( F \circ M \right) \,\horzbar \end{bmatrix} \cdot \begin{bmatrix} \vphantom{\begin{bmatrix} M^*v_0 \\ \vdots \\ M^*v_5 \end{bmatrix}} \mathmakebox[4em][c]{M} \end{bmatrix}^{-1} \cdot \begin{bmatrix} M^* v_0 \\ \vdots \\ M^* v_5\end{bmatrix} = \begin{bmatrix}\horzbar\, \frac{\partial F}{\partial Z_i} \circ M \,\horzbar\end{bmatrix}\cdot \begin{bmatrix} M^* v_0 \\ \vdots \\ M^* v_5\end{bmatrix} = M^* \left( dF(v) \right).
\end{equation*}
To prove the second assertion, we note that $\operatorname{div}(M^* v) = \sum_i \frac{\partial}{\partial Z_i} \left( \sum_j M^{-1}_{ij} \cdot \left( v_j \circ M \right) \right) =\sum_i \sum_j M^{-1}_{ij} \cdot \frac{\partial}{\partial Z_i} \left( v_j \circ M \right) = \sum_i \sum_j M^{-1}_{ij} \cdot \sum_k \left( \left( \frac{\partial v_j}{\partial Z_k} \circ M\right) \cdot M_{ki} \right) = \sum_j \sum_k \left( \sum_i M_{ki} \cdot M^{-1}_{ij} \right) \cdot \frac{\partial v_j}{\partial Z_k} \circ M = M^* \left( \operatorname{div}(v) \right)$; the last equality amounts to $\sum_i M_{ki} \cdot M^{-1}_{ij} = \delta_{kj}$.
\end{proof}
Lemma \ref{pullback} implies that solutions $F$, $A$ and $v$ to \eqref{thing} transport under pullback, since if $F^{q + 1} \mid A - q \cdot dF(v) + F \cdot \operatorname{div}(v)$, then $M^* F^{q + 1} \mid M^* A - q \cdot d\left( M^* F \right)(M^*v) + M^* F\cdot \operatorname{div}(M^* v)$.

\begin{theorem} \label{last}
Let $f(X_0, X_1, X_2)$ be a ternary sextic of Waring rank 3 whose associated plane curve $C \subset \mathbb{P}^2$ is smooth. Then the smooth sextic $X \subset \mathbb{P}^5$ defined by $f(X_0, X_1, X_2) - f(Y_0, Y_1, Y_2)$ is such that the $\iota$-invariant coniveau-1 Hodge substructure $H^4(X, \mathbb{Q})^+ \subset H^4(X, \mathbb{Q})$ is annihilated outside of a divisor $Y \subset X$.
\end{theorem}
\begin{proof}
We fix $q = 2$ and let $A \in R^{12}$ be arbitrary and $\sigma$-anti-invariant. By Lemma \ref{final} and Theorem \ref{final_correctness}, Construction \ref{full_construction} yields a vector field $v$ that solves \eqref{thing} for $A$ and for $F(U_0, U_1, U_2, V_0, V_1, V_2)$ the Fermat sextic. Applying our hypothesis on $f$, we write $f(U_0, U_1, U_2) = L_0(U_0, U_1, U_2)^6 + L_1(U_0, U_1, U_2)^6 + L_2(U_0, U_1, U_2)^6$, for linear forms $L_0$, $L_1$ and $L_2$. Since $C$ is necessarily singular at each projective point where the hyperplanes defined by $L_0$, $L_1$, and $L_2$ intersect, we conclude that $L_0$, $L_1$ and $L_2$ are linearly independent.

Pulling back both $v$ and $A$ along the change of coordinates
\begin{equation*}
L :
\left[ \begin{array}{c}
U_0 \\
U_1 \\
U_2 \\
V_0 \\
V_1 \\
V_2
\end{array} \right]
\mapsto
\left[\begin{array}{c|c}
\begin{array}{@{}c@{}}
\horzbar\,L_0\,\horzbar \\
\horzbar\,L_1\,\horzbar \\
\horzbar\,L_2\,\horzbar
\end{array}
&
\\
\hline
&
\begin{array}{@{}c@{}}
\horzbar\,L_0\,\horzbar \\
\horzbar\,L_1\,\horzbar \\
\horzbar\,L_2\,\horzbar
\end{array}
\end{array}\right] \cdot
\left[ \begin{array}{c}
U_0 \\
U_1 \\
U_2 \\
V_0 \\
V_1 \\
V_2
\end{array} \right],
\end{equation*}
we obtain a solution, namely $L^* v$, of \eqref{thing} for $f(U_0, U_1, U_2) + f(V_0, V_1, V_2)$ and $L^* A$. This solves Question \ref{question} for $f(U_0, U_1, U_2) + f(V_0, V_1, V_2)$, since $L^*$ isomorphically maps the space of $\sigma$-anti-invariant polynomials $A$ to itself (in fact, $L$ and $\sigma$ commute).

Pulling back both $L^* v$ and $L^* A$ along \eqref{change}, we obtain a solution, namely $p^* L^* v$, to \eqref{thing} for $f(X_0, X_1, X_2) - f(Y_0, Y_1, Y_2)$ and for $p^* L^* A$. These latter pullbacks exhaust the space of $\iota$-anti-invariant polynomials of degree $12$, since \eqref{change} identifies $\iota$ and $\sigma$. By the discussion at the beginning of Section \ref{solve}, this fact confirms the prediction made by the generalized Hodge conjecture as regards $H^4(X, \mathbb{Q})^+ \subset H^4(X, \mathbb{Q})$.
\end{proof}

\printbibliography

\end{document}